\documentclass[12pt,a4paper,reqno]{amsart} 
\usepackage{amssymb}
\usepackage{latexsym}
\usepackage{amsmath}
\usepackage{mathrsfs}
\usepackage[X2,T1]{fontenc}
\usepackage[applemac]{inputenc}
\usepackage{cite}
\usepackage{calc}                   

\newcommand{\scal}[2]{\langle #1,#2\rangle}
\newcommand{\rr}[1]{\mathbf R^{#1}}

\newcommand{\nm}[2]{\Vert #1\Vert _{#2}}

\newcommand{\abp}[1]{\vert #1\vert}
\newcommand{\op}{\operatorname{Op}}

\newcommand{\sets}[2]{\{ \, #1\, ;\, #2\, \} }
\newcommand{\Sets}[2]{\left \{ \, #1\, ;\, #2\, \right \} }

\newcommand{\ep}{\varepsilon}
\newcommand{\fy}{\varphi}
\newcommand{\cdo}{\, \cdot \, }

\newcommand{\loc}{\operatorname{loc}}
\newcommand{\wpr}{{\text{\footnotesize $\#$}}}

\newcommand{\vrum}{\vspace{0.1cm}}

\newcommand{\px}{\psi_{x_0}}

\newcommand{\nn}[1]{{\mathbf N}^{#1}}

\newcommand{\maclS}{\mathcal S}

\newcommand{\mascF}{\mathscr F}
\newcommand{\mascS}{\mathscr S}
\newcommand{\mascP}{\mathscr P}

\def\px{\langle x \rangle}

\def\MdR{\mathbf{M}(d, \mathbf{R} )}

\numberwithin{equation}{section}          

\newtheorem{thm}{Theorem}
\numberwithin{thm}{section}
\newtheorem*{tom}{\rubrik}
\newcommand{\rubrik}{}
\newtheorem{prop}[thm]{Proposition}
\newtheorem{cor}[thm]{Corollary}
\newtheorem{lemma}[thm]{Lemma}

\theoremstyle{definition}

\newtheorem{defn}[thm]{Definition}

\theoremstyle{remark}

\author{Ahmed Abdeljawad}

\address{Dipartimento di Matematica, Universit\`a di Torino, Italy}

\email{ahmed.abdeljawad@unito.it}

\author{Marco Cappiello}

\address{Dipartimento di Matematica, Universit\`a di Torino, Italy}

\email{marco.cappiello@unito.it}

\author{Joachim Toft}

\address{Department of Mathematics,
Linn{\ae}us University, V{\"a}xj{\"o}, Sweden}

\email{joachim.toft@lnu.se}

\title{Pseudo-differential calculus in anisotropic Gelfand-Shilov setting}

\begin{document}

\begin{abstract}
We study some classes of pseudo-differential operators with symbols $a$ 
admitting anisotropic exponential growth at infinity 
and we prove 
mapping properties for these operators on Gelfand-Shilov spaces of type $\mathscr{S}$.
%
%
%
%
%
Moreover, we deduce algebraic and certain invariance properties of these classes.
\end{abstract}

\maketitle

\par

\section{Introduction}\label{sec0}

\par

Gelfand-Shilov spaces of type $\mathscr{S}$ have been introduced in the book
\cite{GS} as an alternative functional setting to the Schwartz space $\mathscr{S}(\rr d)$
of smooth and rapidly decreasing functions for Fourier analysis and for the study of
partial differential equations. Namely, fixed $s>0, \sigma >0$, the space
$\mathcal{S}_s^\sigma(\rr d )$ can be defined as the space of all functions
$f \in C^\infty(\rr d)$ satisfying an estimate of the form 
\begin{equation}\label{GSestimate1}
\sup_{\alpha, \beta \in \nn d}\sup_{x \in \rr d} \frac{|x^\beta \partial^\alpha f(x)|}
{h^{|\alpha+\beta|}\alpha!^\sigma \beta!^s} < \infty
\end{equation}
for some constant $h >0$, or the equivalent condition
\begin{equation}\label{GSestimate2}
\sup_{\alpha \in \nn d} \sup_{x \in \rr d} \frac {|e^{r|x|^{\frac 1s}} \partial^\alpha f(x)|}
{h^{|\alpha|} \alpha !^{\sigma}}  
 < \infty
\end{equation}
for some constants $h,r >0$.
For $\sigma >1$,  $\mathcal{S}^\sigma _s(\rr d)$ represents a natural global
counterpart of the Gevrey class $G^\sigma(\rr d)$ but, in addition, the condition
\eqref{GSestimate2} encodes a precise description of the behavior at infinity of
$f$. Together with $\mathcal{S}_s^\sigma (\rr d)$ one can also consider the space
$\Sigma_s^\sigma (\rr d )$, which has been defined in \cite{Pi}  by requiring
\eqref{GSestimate1} (respectively \eqref{GSestimate2}) to hold for every
$\ep >0$ (respectively for every $ h,r >0$). The duals of $\maclS _s
^\sigma (\rr d)$ and $\Sigma _s^\sigma (\rr d)$ and further generalizations
of these spaces have been then introduced in the spirit of Komatsu theory
of ultradistributions, see \cite{ChuChuKim, Pi}. 

\par

After their appearance, Gelfand-Shilov spaces have been recognized as a
natural functional setting for pseudo-differential and Fourier integral operators,
due to their nice behavior under Fourier transformation, and applied in the study
of several classes of partial differential equations, see e.{\,}g.
\cite{AC, Ca1, Ca2, CGR2, CGR3, CPP, CPP2}.

\par

According to the condition on the decay at infinity of the elements of $\maclS
_s^\sigma (\rr d)$ and $\Sigma_s^\sigma (\rr d)$, we can define on these
spaces pseudo-differential operators with symbols $a(x,\xi)$ admitting an
exponential growth at infinity. These operators are commonly known as
operators {\it of infinite order} and they have been studied in \cite{BM} in
the analytic class and in \cite{Za, CZ, KN} in the Gevrey spaces where the
symbol has an exponential growth only with respect to $\xi$ and applied to
the Cauchy problem for hyperbolic and Schr\"odinger equations in Gevrey
classes, see \cite{ CZ, CZ2, KB, CRe1}. Parallel results have been obtained
in Gelfand-Shilov spaces for symbols admitting exponential growth both in
$x$ and $\xi$, see \cite{Ca1, Ca2, CaWa, CPP, CPP2, Pr}.

\par

We stress that the above results concern the non-quasi-analytic iso\-tro\-pic
case $s =\sigma>1.$ In \cite{CaTo} we considered the more
general case $s=\sigma >0$, which is interesting in particular in connection with
Shubin-type pseudo-differential operators, cf. \cite{CGR2, CRT}. Although the
extension of the complete calculus developed in \cite{Ca1,Ca2} in this case is
out of reach due to the lack of compactly supported functions in $\maclS
_s^\sigma (\rr d)$ and $\Sigma_s^\sigma (\rr d)$, nevertheless some interesting
results can be achieved also in this case by using different tools than the usual
micro-local techniques, namely a method based on the use of modulation
spaces and of the short time Fourier transform.

\par

In the present paper, we further
generalize the results of \cite{CaTo} to the case when $s>0$ and $\sigma >0$ may be
different from each other. Thus the symbols we consider may have different rates of
exponential growth and anisotropic Gevrey-type regularity in $x$ and $\xi$.
%
%
More precisely, the symbols should obey conditions of the form
\begin{equation}\label{GSestimate3}
\sup_{\alpha ,\beta \in \nn d} \sup_{x,\xi  \in \rr d} \frac {|e^{-r(|x|^{\frac 1s}+|\xi |^{\frac 1\sigma})}
\partial _x^\alpha \partial _\xi ^\beta a(x,\xi )|}
{h^{|\alpha +\beta|} \alpha !^{\sigma} \beta !^s}  
 < \infty
\end{equation}
for suitable restrictions on the constants $h,r >0$ (cf. \eqref{GSestimate2}). We prove that
if $h>0$, and \eqref{GSestimate3} holds true for every $r>0$, then the pseudo-differential
operator $\op (a)$ is continuous on $\maclS _s^\sigma$ and on $(\maclS _s^\sigma)'$.
If instead $r>0$, and \eqref{GSestimate3} holds true for every $h>0$, then we prove
that $\op (a)$ is continuous on $\Sigma _s^\sigma$ and on $(\Sigma _s^\sigma)'$ (cf.
Theorems \ref{Thm:theorem2}  and  \ref{Thm:theorem1}). We also prove that
pseudo-differential operators with symbols satisfying such conditions form algebras
(cf. Theorems \ref{Thm:GammaAlgebras} and \ref{Composition}). Finally we show
that our span of pseudo-differential operators is invariant under the choice of representation
(cf. Theorem \ref{ThmCalculiTransf2}).

\par

An important ingredient in the analysis which is used to reach these properties concerns
characterizations of symbols above in terms of suitable estimates of their short-time Fourier
transforms. Such characterizations are deduced in Section \ref{sec2}.

\par

The paper is organized as follows. In Section \ref{sec1}, after recalling some basic properties of the
spaces $\mathcal{S}_s^\sigma(\rr d)$ and $\Sigma_s^\sigma(\rr d)$, we introduce several
general symbol classes. In Section \ref{sec2} we characterize these symbols
in terms of the behavior of their short time Fourier transform. In Section
\ref{sec3} we deduce continuity on $\mathcal{S}
_s(\rr d)$ and $\Sigma_s(\rr d)$, composition and invariance properties for
pseudo-differential operators in our classes.

\par

\section{Preliminaries}\label{sec1}

\par

In this section we recall some basic facts, especially
concerning Gel\-fand-Shilov spaces, the
short-time Fourier transform and pseudo-dif\-fe\-ren\-tial operators.

\par

We let $\mathscr{S}(\rr d)$ be the Schwartz space
of rapidly decreasing
functions on $\rr d$ together with their derivatives, and by $\mathscr{S}'(\rr d)$
the corresponding dual space of tempered distributions. Moreover
$\MdR$ will denote the vector space of real $d\times d$ matrices.

\subsection{Gelfand-Shilov spaces} 
\par
We start by recalling some facts about Gelfand-Shilov spaces.
Let $0<h,s,\sigma \in \mathbf R$ be fixed. Then
$\maclS _{s;h}^{\sigma}(\rr d)$ is
the Banach space of all $f\in C^\infty (\rr d)$ such that
\begin{equation}\label{gfseminorm}
\nm f{\maclS _{s;h}^{\sigma}}\equiv \sup_{\alpha ,\beta \in
\mathbf N^d} \sup_{x \in \rr d} \frac {|x^\alpha \partial ^\beta
f(x)|}{h^{|\alpha | + |\beta |}\alpha !^s\, \beta !^\sigma}<\infty,
\end{equation}
endowed with the norm \eqref{gfseminorm}.

\par

The \emph{Gelfand-Shilov spaces} $\maclS _{s}^{\sigma}(\rr d)$ and
$\Sigma _{s}^{\sigma}(\rr d)$ are defined as the inductive and projective 
limits respectively of $\maclS _{s;h}^{\sigma}(\rr d)$. This implies that
\begin{equation}\label{GSspacecond1}
\maclS _{s}^{\sigma}(\rr d) = \bigcup _{h>0}\maclS _{s;h}^{\sigma}(\rr d)
\quad \text{and}\quad \Sigma _{s}^{\sigma}(\rr d) =\bigcap _{h>0}
\maclS _{s;h}^{\sigma}(\rr d),
\end{equation}
and that the topology for $\maclS _{s}^{\sigma}(\rr d)$ is the strongest
possible one such that the inclusion map from $\maclS _{s;h}^{\sigma}
(\rr d)$ to $\maclS _{s}^{\sigma}(\rr d)$ is continuous, for every choice 
of $h>0$. The space $\Sigma _{s}^{\sigma}(\rr d)$ is a Fr{\'e}chet space
with seminorms $\nm \cdo {\maclS _{s;h}^{\sigma}}$, $h>0$. Moreover,
$\Sigma _{s}^{\sigma}(\rr d)\neq \{ 0\}$, if and only if $s+\sigma \ge 1$ and $(s,\sigma )\neq
(\frac 12,\frac 12)$, and
$\maclS _{s}^{\sigma}(\rr d)\neq \{ 0\}$, if and only
if $s+\sigma \ge 1$.

\par

The spaces $\maclS _{s}^{\sigma}(\rr d)$ and $\Sigma _{s}^{\sigma}(\rr d)$ can be characterized 
also in terms of the exponential decay of their elements, namely $f \in \maclS _{s}^{\sigma}(\rr d)$ 
(respectively $f \in \Sigma _{s}^{\sigma}(\rr d)$) if and only if 
$$
|\partial^\alpha f(x)| \lesssim \ep^{|\alpha|} (\alpha!)^\sigma  e^{-h|x|^{\frac 1s}}
$$
for some $h>0, \ep>0$ (respectively for every $h>0, \ep>0$). 
Moreover we recall that for $s <1$ the elements of $\maclS _{s}^{\sigma}(\rr d)$ admit entire 
extensions to $\mathbf{C}^d$ satisfying suitable exponential bounds, cf. \cite{GS} for details.

\medspace

The \emph{Gelfand-Shilov distribution spaces} $(\maclS _{s}^{\sigma})'(\rr d)$
and $(\Sigma _{s}^{\sigma})'(\rr d)$ are the projective and inductive limit
respectively of $(\maclS _{s;h}^{\sigma})'(\rr d)$.  This means that
\begin{equation}\tag*{(\ref{GSspacecond1})$'$}
(\maclS _{s}^{\sigma})'(\rr d) = \bigcap _{h>0}(\maclS _{s;h}^{\sigma})'(\rr d)\quad
\text{and}\quad (\Sigma _{s}^{\sigma})'(\rr d) =\bigcup _{h>0}(\maclS _{s;h}^{\sigma})'(\rr d).
\end{equation}
We remark that in \cite{Pi2} it is proved that $(\maclS _{s}^{\sigma})'(\rr d)$
is the dual of $\maclS _{s,\sigma}(\rr d)$, and $(\Sigma _{s}^{\sigma})'(\rr d)$
is the dual of $\Sigma _{s}^{\sigma}(\rr d)$ (also in topological sense).

\par

For every $s,\sigma >0$ we have
\begin{equation}\label{Eq:GSEmbeddings}
\Sigma _s^\sigma (\rr d)
\hookrightarrow
\maclS _s^\sigma (\rr d)
\hookrightarrow
\Sigma _{s+\ep}^{\sigma +\ep}(\rr d)
\hookrightarrow
\mascS (\rr d)
\end{equation}
for every $\ep >0$. If $s+\sigma \ge 1$, then
the last two inclusions in \eqref{Eq:GSEmbeddings} are dense,
and if in addition $(s,\sigma )\neq (\frac 12,\frac 12)$, then the
first inclusion in \eqref{Eq:GSEmbeddings} is dense.

\par

From these properties it follows that $\mascS '(\rr d)\hookrightarrow
(\maclS _s^\sigma)'(\rr d)$ when $s+\sigma \ge 1$, and if in addition
$(s,\sigma )\neq (\frac 12,\frac 12)$, then $(\maclS _s^\sigma)'(\rr d)
\hookrightarrow (\Sigma _s^\sigma)'(\rr d)$.

\par

The Gelfand-Shilov spaces possess several convenient mapping
properties. For example they are invariant under
translations, dilations, and to some extent tensor products
and (partial) Fourier transformations.

\par

The Fourier transform $\mathscr F$ is the linear and continuous
map on $\mascS (\rr d)$,
given by the formula
$$
(\mathscr Ff)(\xi )= \widehat f(\xi ) \equiv (2\pi )^{-\frac d2}\int _{\rr
{d}} f(x)e^{-i\scal  x\xi }\, dx
$$
when $f\in \mascS (\rr d)$. Here $\scal \cdo \cdo$ denotes the usual
scalar product on $\rr d$. 
The Fourier transform extends  uniquely to homeomorphisms
from $(\maclS _{s}^{\sigma})'(\rr d)$ to $(\maclS _{\sigma}^{s})'(\rr d)$,
and from  $(\Sigma _{s}^{\sigma})'(\rr d)$ to $(\Sigma _{\sigma}^{s})'(\rr d)$.
Furthermore, it restricts to homeomorphisms from
$\maclS _{s}^{\sigma}(\rr d)$ to $\maclS _{\sigma}^{s}(\rr d)$,
and from  $\Sigma _{s}^{\sigma}(\rr d)$ to $\Sigma _{\sigma}^{s}(\rr d)$.

\par

\medspace

Some considerations later on involve a broader family of
Gelfand-Shilov spaces. More precisely, for $s_j,\sigma _j\in \mathbf R_+$,
$j=1,2$, the Gelfand-Shilov spaces $\maclS _{s _1,s_2}^{\sigma _1,\sigma _2}(\rr {d_1+d_2})$ and
$\Sigma _{s _1,s_2}^{\sigma _1,\sigma _2}(\rr {d_1+d_2})$ consist of all functions
$F\in C^\infty (\rr {d_1+d_2})$ such that
\begin{equation}\label{GSExtCond}
|x_1^{\alpha _1}x_2^{\alpha _2}\partial _{x_1}^{\beta _1}
\partial _{x_2}^{\beta _2}F(x_1,x_2)| \lesssim
h^{|\alpha _1+\alpha _2+\beta _1+\beta _2|}
\alpha _1!^{s_1}\alpha _2!^{s_2}\beta _1!^{\sigma _1}\beta _2!^{\sigma _2}
\end{equation}
for some $h>0$ respective for every $h>0$. The  topologies, and the duals
\begin{alignat*}{3}
&(\maclS _{s _1,s_2}^{\sigma _1,\sigma _2})'(\rr {d_1+d_2}) &
&\quad \text{and} \quad &
&(\Sigma _{s _1,s_2}^{\sigma _1,\sigma _2})'(\rr {d_1+d_2})
\intertext{of}
&\maclS _{s _1,s_2}^{\sigma _1,\sigma _2}(\rr {d_1+d_2}) &
&\quad \text{and} \quad &
&\Sigma _{s _1,s_2}^{\sigma _1,\sigma _2}(\rr {d_1+d_2}),
\end{alignat*}
respectively, and their topologies
are defined in analogous ways as for the spaces $\maclS _s^\sigma (\rr d)$
and $\Sigma _s^\sigma (\rr d)$ above.

\par

The following proposition explains mapping properties of partial
Fourier transforms on Gelfand-Shilov spaces, and follows by similar
arguments as in analogous situations in
\cite{GS}. The proof is therefore omitted. Here, $\mascF _1F$
and $\mascF _2F$ are the partial
Fourier transforms of $F(x_1,x_2)$ with respect to
$x_1\in \rr {d_1}$ and $x_2\in \rr {d_2}$,
respectively.

\par

\begin{prop}\label{propBroadGSSpaceChar}
Let $s_j,\sigma _j >0$, $j=1,2$.
Then the following is true:
\begin{enumerate}
\item the mappings $\mascF _1$ and $\mascF _2$ on $\mascS (\rr {d_1+d_2})$
restrict to homeomorphisms
\begin{align*}
\mascF _1 \, &: \, \maclS _{s _1,s_2}^{\sigma _1,\sigma _2}(\rr {d_1+d_2}) \to
\maclS _{\sigma _1,s_2}^{s_1,\sigma _2}(\rr {d_1+d_2})
\intertext{and}
\mascF _2 \, &: \, \maclS _{s _1,s_2}^{\sigma _1,\sigma _2}(\rr {d_1+d_2}) \to
\maclS _{s _1,\sigma _2}^{\sigma _1,s_2}(\rr {d_1+d_2})
\text ;
\end{align*}

\vrum

\item the mappings $\mascF _1$ and $\mascF _2$ on
$\mascS (\rr {d_1+d_2})$ are uniquely extendable to
homeomorphisms
\begin{align*}
\mascF _1 \, &: \, (\maclS _{s _1,s_2}^{\sigma _1,\sigma _2})'(\rr {d_1+d_2}) \to
(\maclS _{\sigma _1,s_2}^{s_1,\sigma _2})'(\rr {d_1+d_2})
\intertext{and}
\mascF _2 \, &: \, (\maclS _{s _1,s_2}^{\sigma _1,\sigma _2})'(\rr {d_1+d_2}) \to
(\maclS _{s _1,\sigma _2}^{\sigma _1,s_2})'(\rr {d_1+d_2}).
\end{align*}
\end{enumerate}

\par

The same holds true if the $\maclS  _{s _1,s_2}^{\sigma _1,\sigma _2}$-spaces and
their duals are replaced by
corresponding $\Sigma  _{s _1,s_2}^{\sigma _1,\sigma _2}$-spaces and their duals.
\end{prop}

\par

The next two results follow from \cite{ChuChuKim}. The proofs are therefore omitted.

\begin{prop}
Let $s_j,\sigma _j> 0$, $j=1,2$. Then the following
conditions are equivalent.
\begin{enumerate}
\item $F\in \maclS _{s _1,s_2}^{\sigma _1,\sigma _2}(\rr {d_1+d_2})$\quad
($F\in \Sigma _{s _1,s_2}^{\sigma _1,\sigma _2}(\rr {d_1+d_2})$);

\vrum

\item for some $h>0$ (for every $h>0$) it holds
\begin{equation*}
\displaystyle{|F(x_1,x_2)|\lesssim e^{-h(|x_1|^{\frac 1{s_1}} + |x_2|^{\frac 1{s_2}} )}}
\quad \text{and}\quad 
\displaystyle{|\widehat F(\xi _1,\xi _2)|\lesssim
e^{-h(|\xi _1|^{\frac 1{\sigma _1}} + |\xi _2|^{\frac 1{\sigma _2}} )}}.
\end{equation*}
\end{enumerate}
\end{prop}

\par

We notice that if
$s_j+\sigma _j<1$ for some $j=1,2$, then
$\maclS _{s _1,s_2}^{\sigma _1,\sigma _2}(\rr {d_1+d_2})$
and $\Sigma _{s _1,s_2}^{\sigma _1,\sigma _2}(\rr {d_1+d_2})$
are equal to the trivial space $\{ 0\}$.
Likewise, if $s_j=\sigma _j=\frac 12$ for some $j=1,2$, then
$\Sigma _{s _1,s_2}^{\sigma _1,\sigma _2}(\rr {d_1+d_2}) = \{ 0\}$.

\par

\subsection{The short time Fourier transform and Gelfand-Shilov spaces}

\par

We recall here some basic facts about
the short-time Fourier transform and weights.

\par

Let $\phi \in \maclS _s^\sigma (\rr d)\setminus 0$ be fixed. Then the short-time
Fourier transform of $f\in (\maclS _s^\sigma )'(\rr d)$ is given by
$$
(V_\phi f)(x,\xi ) = (2\pi )^{-\frac d2}(f,\phi (\cdo -x)
e^{i\scal \cdo \xi})_{L^2}.
$$
Here $(\cdo ,\cdo )_{L^2}$ is the unique extension of the $L^2$-form on
$\maclS _s^\sigma (\rr d)$ to a continuous sesqui-linear form on $(\maclS
_s^\sigma )'(\rr d)\times \maclS _s^\sigma (\rr d)$. In the case
$f\in L^p(\rr d)$, for some $p\in [1,\infty]$, then $V_\phi f$ is given by
$$
V_\phi f(x,\xi ) \equiv (2\pi )^{-\frac d2}\int _{\rr d}f(y)\overline{\phi (y-x)}
e^{-i\scal y\xi}\, dy .
$$

\par

The following characterizations of the
$\maclS _{s_1,s_2}^{\sigma _1,\sigma _2}(\rr {d_1+d_2})$,
$\Sigma _{s_1,s_2}^{\sigma _1,\sigma _2}(\rr {d_1+d_2})$
and their duals
follow by similar arguments as in the proofs of
Propositions 2.1 and 2.2 in \cite{To15}. The details are left
for the reader.

\par

\begin{prop}\label{Prop:STFTGelfand2}
Let $s_j,\sigma _j>0$ be such that $s_j+\sigma _j\ge 1$, $j=1,2$, 
$s_0\le s$ and $\sigma_0\le \sigma$. Also let
$\phi \in \mathcal S_{s_1,s_2}^{\sigma _1,\sigma _2}(\rr {d_1+d_2})
\setminus 0$ and let $f$ be a Gelfand-Shilov distribution on
$\rr {d_1+d_2}$.
Then the following is true:
\begin{enumerate}
\item $f\in  \maclS _{s_1,s_2}^{\sigma _1,\sigma _2}(\rr {d_1+d_2})$,
if and only if
\begin{equation}\label{stftexpest2}
|V_\phi f(x_1,x_2,\xi _1,\xi _2)|
\lesssim
e^{-r (|x_1|^{\frac 1{s_1}} + |x_2|^{\frac 1{s_2}}
+|\xi _1|^{\frac 1{\sigma _1}} +|\xi _2|^{\frac 1{\sigma _2}} )},
\end{equation}
holds for some $r > 0$;

\vrum

\item if in addition
$\phi \in \Sigma _{s_1,s_2}^{\sigma _1,\sigma _2}(\rr {d_1+d_2})
\setminus 0$, then 
$f\in  \Sigma _{s_1,s_2}^{\sigma _1,\sigma _2}(\rr {d_1+d_2})$
if and only if
\begin{equation}\label{stftexpest2A}
|V_\phi f(x_1,x_2,\xi _1,\xi _2)|
\lesssim
e^{-r (|x_1|^{\frac 1{s_1}} + |x_2|^{\frac 1{s_2}}
+|\xi _1|^{\frac 1{\sigma _1}} +|\xi _2|^{\frac 1{\sigma _2}} )}
\end{equation}
holds for every $r > 0$.
\end{enumerate}
\end{prop}

\par

A proof of Proposition  \ref{Prop:STFTGelfand2} can be found in
e.{\,}g. \cite{GZ} (cf. \cite[Theorem 2.7]{GZ}). The
corresponding result for Gelfand-Shilov distributions
is the following improvement of \cite[Theorem 2.5]{To11}.

\par

\begin{prop}\label{Prop:STFTGelfand2Dist}
Let $s_j,\sigma _j>0$ be such that $s_j+\sigma _j\ge 1$, $j=1,2$, 
$s_0\le s$ and $t_0\le t$. Also let
$\phi \in \mathcal S_{s_1,s_2}^{\sigma _1,\sigma _2}(\rr {d_1+d_2})
\setminus 0$ and let $f$ be a Gelfand-Shilov distribution on
$\rr {d_1+d_2}$.
Then the following is true:
\begin{enumerate}
\item $f\in  (\maclS _{s_1,s_2}^{\sigma _1,\sigma _2})'(\rr {d_1+d_2})$,
if and only if
\begin{equation}\label{stftexpest2Dist}
|V_\phi f(x_1,x_2,\xi _1,\xi _2)|
\lesssim
e^{r (|x_1|^{\frac 1{s_1}} + |x_2|^{\frac 1{s_2}}
+|\xi _1|^{\frac 1{\sigma _1}} +|\xi _2|^{\frac 1{\sigma _2}} )}
\end{equation}
holds for every $r > 0$;

\vrum

\item if in addition
$\phi \in \Sigma _{s_1,s_2}^{\sigma _1,\sigma _2}(\rr {d_1+d_2})
\setminus 0$, then 
$f\in  (\Sigma _{s_1,s_2}^{\sigma _1,\sigma _2})'(\rr {d_1+d_2})$,
if and only if
\begin{equation}\label{stftexpest2DistA}
|V_\phi f(x_1,x_2,\xi _1,\xi _2)|
\lesssim
e^{r (|x_1|^{\frac 1{s_1}} + |x_2|^{\frac 1{s_2}}
+|\xi _1|^{\frac 1{\sigma _1}} +|\xi _2|^{\frac 1{\sigma _2}} )}
\end{equation}
holds for some $r > 0$.
\end{enumerate}
\end{prop}

\par

A function $\omega$ on $\rr d$ is called a \emph{weight} or
\emph{weight function},
if $\omega ,1/\omega \in L^\infty _{\loc} (\rr d)$
are positive everywhere. It is often assumed that $\omega$ is $v$-moderate
for some positive function $v$ on $\rr d$. This means that
\begin{equation}\label{vModerate}
\omega (x+y)\lesssim \omega (x)v(y),\quad x,y\in \rr d.
\end{equation}
If $v$ is even and satisfies \eqref{vModerate} with $\omega =v$,
then $v$ is called submultiplicative. 

\par

For any $s>0$, let
$\mascP _s(\rr d)$ ($\mascP _s^0(\rr d)$) be the set of all
weights $\omega$ on $\rr d$ such that
$$
e^{-r|x|^{\frac 1s}}\lesssim \omega (x)\lesssim e^{r|x|^{\frac 1s}}
$$
for some $r>0$ (for every $r>0$). In similar ways, if $s,\sigma >0$, then
$\mascP _{s,\sigma}(\rr {2d})$ ($\mascP _{s,\sigma}^0(\rr {2d})$) consists of all
submultiplicative weight functions $\omega$ on $\rr {2d}$ such that
$$
e^{-r(|x|^{\frac 1s}+|\xi |^{\frac 1\sigma})}\lesssim \omega (x,\xi )
\lesssim e^{r(|x|^{\frac 1s}+|\xi |^{\frac 1\sigma})}
$$
for some $r>0$ (for every $r>0$).
In particular, if $\omega \in \mascP _{s,\sigma}(\rr {2d})$ ($\mascP _{s,\sigma}^0(\rr {2d})$), then 
\begin{equation} \label{estomega}
\omega (x+y,\xi +\eta ) \lesssim \omega (x,\xi )e^{r(|y|^{\frac 1s}
+|\eta |^{\frac 1\sigma})}, \quad  x,y,\xi ,\eta  \in \rr d,
\end{equation}
for some $r>0$ (for every $r>0$).

\par

\par

\subsection{Pseudo-differential operators}

\par

Let  $A \in \MdR$ and $s\ge \frac 12$ be fixed, and let $a\in \maclS _{s}
(\rr {2d})$. Then the
\emph{pseudo-differential operator} $\op _A(a)$
with symbol $a$ is the continuous operator on $\maclS _{s} (\rr d)$,
defined by the formula
\begin{equation}\label{e0.5}
(\op _A(a)f)(x)
=
(2\pi  ) ^{-d}\iint a(x-A(x-y),\xi )f(y)e^{i\scal {x-y}\xi }\,
dyd\xi .
\end{equation}
We set $\op _t(a)=\op _A(a)$ when $t\in \mathbf R$, $A=t\cdot I$
and $I$ is the identity matrix, and notice that this definition agrees with
the Shubin type pseudo-differential operators (cf. e.{\,}g. \cite{To7}).

\par

If instead $a\in
(\maclS _{s,\sigma}^{\sigma ,s})'(\rr {2d})$, then $\op _A(a)$ is
defined to be the continuous
operator from $\maclS _{s}^\sigma (\rr d)$ to
$(\maclS _{s}^\sigma )'(\rr d)$ with
the kernel in $(\maclS _{s}^\sigma )'(\rr {2d})$, given by
\begin{equation}\label{KaADef}
K_{a,A}(x,y) \equiv (\mascF _2^{-1}a)(x-A(x-y),x-y).
\end{equation}
It is easily seen that the latter definition agrees with \eqref{e0.5}$'$
when $a\in L^1(\rr {2d})$.

\par

If $t=\frac 12$, then $\op _t(a)$ is equal to the Weyl
operator $\op ^w(a)$ for $a$. If instead $t=0$, then the standard
(Kohn-Nirenberg) representation $a(x,D)$ is obtained.

\par

\subsection{Symbol classes}

\par

Next we introduce function spaces related to symbol classes
of the pseudo-differential operators. These functions should obey various
conditions of the form
\begin{align}
|\partial _x^\alpha \partial _\xi ^\beta a(x,\xi )|
&\lesssim
h ^{|\alpha +\beta |}\alpha !^\sigma \beta !^s  \omega (x,\xi ),
\label{Eq:symbols2}
\end{align}
for functions on the phase space $\rr {2d}$. For this reason we consider
semi-norms of the form
\begin{equation}\label{Eq:GammaomegaNorm}
\nm a{\Gamma _{(\omega )}^{\sigma ,s;h}}
\equiv \sup _{\alpha ,\beta \in \nn d}
\left (
\sup _{x,\xi \in \rr d} \left (
\frac {|\partial _x^\alpha \partial _\xi ^\beta a(x,\xi )|}{
h ^{|\alpha +\beta |}\alpha !^\sigma \beta !^s \omega (x,\xi )}
\right ) \right ) ,
\end{equation}
indexed by $h>0$,

\par

\begin{defn}\label{Def:GammaSymb2}
Let $s$, $\sigma$ and $h$ be positive constants,
let $\omega$ be a weight on $\rr {2d}$, and let
$$
\omega _r(x,\xi )\equiv e^{r(|x|^{\frac 1s} + |\xi |^{\frac 1\sigma })}   .
$$ 
\begin{enumerate}
\item The set $\Gamma _{(\omega )}^{\sigma ,s;h}  (\rr {2d})$
consists of
all $a \in C^\infty(\rr {2d})$ such that
$\nm a{\Gamma _{(\omega )}^{\sigma ,s;h}}$ in
\eqref{Eq:GammaomegaNorm} is finite.
The set $\Gamma _0^{\sigma ,s;h}  (\rr {2d})$ consists of
all $a \in C^\infty(\rr {2d})$ such that $\nm a{\Gamma _{(\omega_r )}^{\sigma ,s;h}}$ is finite
for every $r>0$, and the topology is the projective limit topology of
$\Gamma _{(\omega _r)}^{\sigma ,s;h}  (\rr {2d})$ with respect to $r>0$;

\vrum

\item The sets $\Gamma _{(\omega )}^{\sigma ,s}  (\rr {2d})$ and
$\Gamma _{(\omega )}^{\sigma ,s;0}  (\rr {2d})$ are given by
$$
\Gamma _{(\omega )}^{\sigma ,s}  (\rr {2d})
\equiv
\bigcup _{h>0}\Gamma _{(\omega )}^{\sigma ,s;h}  (\rr {2d})
\quad \text{and}\quad
\Gamma _{(\omega )}^{\sigma ,s;0}  (\rr {2d})
\equiv
\bigcap _{h>0}\Gamma _{(\omega )}^{\sigma ,s;h}  (\rr {2d}),
$$
and their topologies are the inductive respective the projective topologies
of $\Gamma _{(\omega )}^{\sigma ,s;h}  (\rr {2d})$ with respect to $h>0$.
\end{enumerate}
\end{defn}

\par

Furthermore we have the following classes.

\par

\begin{defn}\label{Def:ExtGSSymbClasses}
For $s_j,\sigma_j\ge 0$, $j=1,2$, and $h,r>0$ 
and $f\in C^\infty(\rr {d_1+d_2})$, let

\begin{equation}\label{symbols}
   \nm{f}{(h,r)}\equiv \sup \left (\frac{\abp{\partial _{x_1}
    ^{\alpha_1}\partial _{x_2}^{\alpha_2}f(x_1,x_2)}}
    {h^{\abp{\alpha_1+\alpha_2}}\alpha_1!^{\sigma _1}
    \alpha_2!^{\sigma _2}e^{r(\abp{x_1}^{\frac 1{s_1}}
    +\abp{x_2}^{\frac 1{s_2}})}} \right ), 
\end{equation}
where the supremum is taken over all $\alpha_1\in \nn {d_1},\alpha_2
\in \nn {d_2},x_1\in \rr {d_1}$ and $x_2\in \rr {d_2}$.
\begin{enumerate}
    \item $\Gamma _{s_1,s_2} ^{\sigma_1,\sigma_2} (\rr {d_1+d_2})$
    consists of all $f\in C^\infty(\rr {d_1+d_2})$ such that 
    $\nm{f}{(h,r)}$ is finite for some $h,r>0$;    
    
    \vrum
    
    \item $\Gamma _{s_1,s_2;0} ^{\sigma_1,\sigma_2} (\rr {d_1+d_2})$
    consists of all $f\in C^\infty(\rr {d_1+d_2})$ such that for some
    $h>0$,
    $\nm{f}{(h,r)}$ is finite for every $r>0$;
    
    \vrum
    
    \item $\Gamma _{s_1,s_2} ^{\sigma_1,\sigma_2;0} (\rr {d_1+d_2})$
    consists of all $f\in C^\infty(\rr {d_1+d_2})$ such that for some
    $r>0$, $\nm{f}{(h,r)}$ is finite for every $h>0$;
    
    \vrum
    
    \item $\Gamma _{s_1,s_2;0} ^{\sigma_1,\sigma_2;0} (\rr {d_1+d_2})$ 
    consists of all
    $f\in C^\infty(\rr {d_1+d_2})$ such that $\nm{f}{(h,r)}$ is
    finite for every $h,r>0$.
\end{enumerate}
\end{defn}

\par

In order to define suitable topologies of the spaces in Definition
\ref{Def:ExtGSSymbClasses}, let $(\Gamma _{s_1,s_2} ^{\sigma_1,\sigma_2})_{(h,r)}
(\rr {d_1+d_2})$ be the set of
$f\in C^\infty(\rr {d_1+d_2})$ such that $\nm{f}{(h,r)}$ is finite. Then
$(\Gamma _{s_1,s_2} ^{\sigma_1,\sigma_2})_{(h,r)} (\rr {d_1+d_2})$ is a
Banach space, and the sets in Definition \ref{Def:ExtGSSymbClasses} are given
by
        \begin{align*}
            \Gamma _{s_1,s_2} ^{\sigma_1,\sigma_2} (\rr {d_1+d_2})
            &=
            \bigcup_{h,r>0}(\Gamma _{s_1,s_2}
            ^{\sigma_1,\sigma_2})_{(h,r)} (\rr {d_1+d_2});
            \\
            \Gamma _{s_1,s_2;0} ^{\sigma_1,\sigma_2} (\rr {d_1+d_2})
            &=
            \bigcup_{h>0}\left(\bigcap_{r>0}
            (\Gamma _{s_1,s_2} ^{\sigma_1,\sigma_2})_{(h,r)} (\rr {d_1+d_2})\right);
            \\
            \Gamma _{s_1,s_2} ^{\sigma_1,\sigma_2;0} (\rr {d_1+d_2})
            &=
            \bigcup_{r>0}\left(\bigcap_{h>0}(\Gamma _{s_1,s_2} 
            ^{\sigma_1,\sigma_2})_{(h,r)} (\rr {d_1+d_2})\right)
            \intertext{ and }
            \Gamma _{s_1,s_2;0} ^{\sigma_1,\sigma_2;0} (\rr {d_1+d_2})
            &=
            \bigcap_{h,r>0}(\Gamma _{s_1,s_2}
            ^{\sigma_1,\sigma_2})_{(h,r)} (\rr {d_1+d_2}),
        \end{align*}
and we equip these spaces by suitable mixed inductive and projective limit topologies
of $(\Gamma _{s_1,s_2} ^{\sigma_1,\sigma_2})_{(h,r)} (\rr {d_1+d_2})$.

\par

In Appendix A we show some further continuity results of the symbol classes
in Definition \ref{Def:ExtGSSymbClasses}.

\par

\section{The short-time Fourier transform and
regularity}\label{sec2}

\par

In this section we deduce equivalences between conditions on
the short-time Fourier transforms of functions or distributions and estimates
on derivatives. 

\par

In what follows we let $\kappa$ be defined as
\begin{equation} \label{kappadef}
\kappa (r) =
\begin{cases}
1\quad &\text{when}\ r\le 1,
\\[1ex]
2^{r-1}\quad &\text{when}\ r> 1.
\end{cases}
\end{equation}

\par

In the sequel we shall frequently use the well known inequality
$$
|x+y|^{\frac 1s} \leq \kappa(s^{-1}) (|x|^{\frac 1s} + |y|^{\frac 1s}), \qquad s >0,\quad
x,y\in \rr d.
$$

\par

\begin{prop}\label{prop1}
Let $s,\sigma >0$ be such that $s+\sigma \ge 1$ and $(s,\sigma )\neq (\frac 12,\frac 12)$,
$\phi \in \Sigma _s^\sigma (\rr d)\setminus 0$, $r>0$ and let
$f$ be a Gelfand-Shilov distribution on $\rr d$. Then the following is true:
\begin{enumerate}
\item If $f\in C^\infty (\rr d)$ and satisfies
\begin{equation}\label{GelfRelCond1}
|\partial ^\alpha f(x)|\lesssim h ^{|\alpha |}\alpha !^\sigma e^{r|x|^{\frac 1s}},
\end{equation}
for every $h >0$ (resp. for some $h >0$), then
\begin{equation}\label{stftcond1}
|V_\phi f(x,\xi )|\lesssim e^{\kappa (s^{-1})r|x|^{\frac 1s}-h |\xi |^{\frac 1\sigma }},
\end{equation}
for every $h >0$ (resp. for some new $h>0$);

\vrum

\item If 
\begin{equation}\label{GelfRelCond1A}
|V_\phi f(x,\xi )|\lesssim e^{r|x|^{\frac 1s}-h |\xi |^{\frac 1\sigma}},
\end{equation}
for every $h>0$ (resp. for some $h>0$), then $f\in C^\infty (\rr d)$ and satisfies
$$
|\partial ^\alpha f(x)|\lesssim h ^{|\alpha |}\alpha !^\sigma e^{\kappa (s^{-1})r|x|^{\frac 1s}},
$$
for every $h>0$ (resp. for some new $h>0$).
\end{enumerate}
\end{prop}

\par


\par

\begin{proof}
We only prove the assertion when \eqref{GelfRelCond1} or \eqref{GelfRelCond1A}
are true for every $h>0$, leaving the straight-forward modifications of the
other cases to the reader. 

\par

Assume that \eqref{GelfRelCond1} holds. Then for every $x \in \rr d$ the function
$$
y\mapsto F_x(y)\equiv f(y+x)\overline{\phi (y)}
$$
belongs to $\Sigma _s^\sigma(\rr d)$, and
$$
|\partial _y^\alpha F_x(y)|
\lesssim
h^{|\alpha |}\alpha !^\sigma 
e^{\kappa (s^{-1})r|x|^{\frac 1s}}e^{-r_0|y|^{\frac 1s}},
$$
for every $h , r_0 >0$. In particular,
\begin{equation}\label{FxEsts}
|F_x(y)|\lesssim e^{\kappa (s^{-1})r|x|^{\frac 1s}}
e^{-r_0|y|^{\frac 1s}}
\quad \text{and}\quad
|\widehat F_x(\xi )|\lesssim e^{\kappa (s^{-1})r|x|^{\frac 1s}}
e^{-r_0|\xi |^{\frac 1\sigma}},
\end{equation}
for every $r_0>0$. Since $|V_\phi f(x,\xi )| = |\widehat F_x(\xi )|$, the estimate
\eqref{stftcond1} follows from the second inequality in \eqref{FxEsts}, and (1) follows.

\par

Next we prove (2). By the inversion formula we get
\begin{equation}\label{Eq:STFTInversionFormula}
f(x) = (2\pi)^{-\frac d2}
\nm \phi{L^2}^{-2} \iint _{\rr {2d}} V_\phi f(y,\eta )
\phi (x-y)e^{i\scal x\eta}\, dyd\eta .
\end{equation}
Here we notice that
\begin{align*}
(x,y,\eta ) &\mapsto V_\phi f(y,\eta )\phi (x-y)e^{i\scal x\eta}
\intertext{is smooth and}
(y,\eta ) &\mapsto \eta ^\alpha V_\phi f(y,\eta )\partial ^\beta \phi
(x-y)e^{i\scal x\eta}
\end{align*}
is an integrable function for every $x$, $\alpha$ and $\beta$, giving that $f$ in
\eqref{Eq:STFTInversionFormula} is smooth.

\par

By differentiation and the fact that $\phi \in \Sigma _s^\sigma$ we get
\begin{multline*}
|\partial ^\alpha f(x)| \asymp \left |
\sum _{\beta \le \alpha} {\alpha \choose \beta} i^{|\beta|}
\iint _{\rr {2d}} \eta ^\beta V_\phi f(y,\eta )  (\partial ^{\alpha -\beta }\phi )(x-y)
e^{i\scal x\eta}\, dyd\eta
\right |
\\[1ex]
\le
\sum _{\beta \le \alpha} {\alpha \choose \beta}
\iint _{\rr {2d}} |\eta ^\beta V_\phi f(y,\eta )  (\partial ^{\alpha -\beta }\phi )(x-y)|
\, dyd\eta
\\[1ex]
\lesssim
\sum _{\beta \le \alpha} {\alpha \choose \beta}
\iint _{\rr {2d}} |\eta ^\beta e^{r |y|^{\frac 1s}}e^{-h |\eta |^{\frac 1\sigma}} 
(\partial ^{\alpha -\beta }\phi )(x-y)|
\, dyd\eta
\\[1ex]
\lesssim
\sum _{\beta \le \alpha} {\alpha \choose \beta}
h_2 ^{|\alpha - \beta |} (\alpha -\beta )!^\sigma
\iint _{\rr {2d}} |\eta ^\beta |e^{-h |\eta |^{\frac 1\sigma}}e^{r |y|^{\frac 1s}}
e^{-h_1|x-y|^{\frac 1s}} \, dyd\eta ,
\end{multline*}
for every $h_1>0$ and $h_2>0$. Since
\begin{equation}\label{Eq:StirlingTypeRel}
|\eta ^\beta e^{-h |\eta |^{\frac 1\sigma }}|\lesssim h_2^{|\beta|}(\beta !)^\sigma
e^{-\frac h2 \cdot  |\eta |^{\frac 1\sigma}},
\end{equation}
we get
\begin{multline}\label{Eq:GSTypeEst1}
|\partial ^\alpha f(x)|
\\[1ex]
\lesssim h_2 ^{|\alpha |}
\sum _{\beta \le \alpha} {\alpha \choose \beta}
(\beta !
(\alpha -\beta )!)^\sigma \iint _{\rr {2d}} e^{-\frac h2 \cdot  |\eta |^{\frac 1\sigma}}
e^{r |y|^{\frac 1s}}
e^{-h_1|x-y|^{\frac 1s}} \, dyd\eta
\\[1ex]
\lesssim (2^{1-s}h_2)^{|\alpha |}(\alpha !)^\sigma \int _{\rr n} e^{r |y|^{\frac 1s}}
e^{-h_1|x-y|^{\frac 1s}} \, dy .
\end{multline}

\par

Since $|y|^{\frac 1s}\le \kappa (s^{-1})(|x|^{\frac 1s}+|y-x|^{\frac 1s})$
and $h_1$ can be chosen
arbitrarily large, it follows from the last estimate that
$$
|\partial ^\alpha f(x)| \lesssim (2h_2)^{|\alpha |}(\alpha !)^\sigma
e^{r \kappa (s^{-1})|x|^{\frac 1s}},
$$
for every $h_2>0$. This gives the result.
\end{proof}

\par

By similar arguments we get the following result. The details
are left for the reader.

\par

\renewcommand{\rubrik}{Proposition \ref{prop1}$'$}

\par

\begin{tom}
Let $s_j,\sigma _j>0$ be such that $s_j+\sigma _j\ge 1$
and $(s_j,\sigma _j)\neq (\frac 12,\frac 12)$, $j=1,2$,
$\phi \in \Sigma _{s_1,s_2}^{\sigma _1,\sigma _2}
(\rr {d_1+d_2})\setminus 0$, $r>0$ and let
$f$ be a Gelfand-Shilov distribution on $\rr {d_1+d_2}$.
Then the following is true:
\begin{enumerate}
\item If $f\in C^\infty (\rr {d_1+d_2})$ and satisfies
\begin{equation}\tag*{(\ref{GelfRelCond1})$'$}
|\partial _{x_1}^{\alpha _1}\partial _{x_2}^{\alpha _2}
f(x_1,x_2)|\lesssim h ^{|\alpha _1+\alpha _2|}
\alpha _1!^{\sigma _1}\alpha _2!^{\sigma _2}
e^{r(|x_1|^{\frac 1{s_1}}+|x_2|^{\frac 1{s_2}}) },
\end{equation}
for every $h >0$ (resp. for some $h >0$), then
\begin{equation}\tag*{(\ref{stftcond1})$'$}
|V_\phi f(x_1,x_2,\xi _1,\xi _2)|
\lesssim e^{\kappa (s_1^{-1})r|x_1|^{\frac 1{s_1}}
+\kappa (s_2^{-1})r|x_2|^{\frac 1{s_2}}
-h (|\xi _1|^{\frac 1{\sigma _1}} +
|\xi _2|^{\frac 1{\sigma _2}}) },
\end{equation}
for every $h >0$ (resp. for some new $h>0$);

\vrum

\item If 
\begin{equation}\tag*{(\ref{GelfRelCond1A})$'$}
|V_\phi f(x_1,x_2,\xi _1,\xi _2)|
\lesssim e^{r(|x_1|^{\frac 1{s_1}}
+|x_2|^{\frac 1{s_2}})
-h (|\xi _1|^{\frac 1{\sigma _1}} +
|\xi _2|^{\frac 1{\sigma _2}}) },
\end{equation}
for every $h>0$ (resp. for some $h>0$), then
$f\in C^\infty (\rr {d_1+d_2})$ and satisfies
$$
|\partial _{x_1}^{\alpha _1}\partial _{x_2}^{\alpha _2}
f(x_1,x_2)|\lesssim h ^{|\alpha _1+\alpha _2|}
\alpha _1!^{\sigma _1}\alpha _2!^{\sigma _2}
e^{\kappa (s_1^{-1})r|x_1|^{\frac 1{s_1}}
+
\kappa (s_2^{-1})r|x_2|^{\frac 1{s_2}} },
$$
for every $h>0$ (resp. for some new $h>0$).
\end{enumerate}
\end{tom}

\par

\par

As a consequence of the previous result we get the following.

\par

\begin{prop}\label{prop2}
Let $s_j,\sigma _j>0$ be such that $s_j+\sigma _j\ge 1$ and
$(s_j,\sigma _j)\neq (\frac 12,\frac 12)$, $j=1,2$,
$\phi \in \Sigma _{s_1,s_2}^{\sigma _1,\sigma _2}
(\rr {d_1+d_2})\setminus 0$ and let
$f$ be a Gelfand-Shilov distribution on $\rr {d_1+d_2}$. 
Then the following is true:
\begin{enumerate}
\item there exist $h,r>0$ such that \eqref{GelfRelCond1A}$'$
holds, if and only if $f\in \Gamma _{s_1,s_2}^{\sigma _1,\sigma _2}
(\rr {d_1+d_2})$;

\vrum

\item there exists $r>0$ such that \eqref{GelfRelCond1A}$'$
holds for every $h>0$, if and only if $f\in
\Gamma _{s_1,s_2}^{\sigma _1,\sigma _2;0} (\rr {d_1+d_2})$;

\vrum

\item \eqref{GelfRelCond1A}$'$ holds
for every $h,r>0$, if and only if $f\in
\Gamma _{s_1,s_2;0}^{\sigma _1,\sigma _2;0} (\rr {d_1+d_2})$.
\end{enumerate}
\end{prop}

\par




\par

By similar arguments that led to Proposition \ref{prop2}
we also get the following. The details are left for the reader.

\par

\begin{prop}\label{prop2'}
Let $s_j,\sigma _j>0$ be such that $s_j+\sigma _j\ge 1$,
$j=1,2$, $\phi \in \maclS
_{s_1,s_2}^{\sigma _1,\sigma _2}
(\rr {d_1+d_2})\setminus 0$ and let
$f$ be a Gelfand-Shilov distribution on $\rr {d_1+d_2}$. 
Then there exists $h>0$ such that \eqref{GelfRelCond1A}$'$
holds for every $r>0$, if and only if $f\in
\Gamma _{s_1,s_2;0}^{\sigma _1,\sigma _2} (\rr {d_1+d_2})$.
\end{prop}

\par

We also have the following version of Proposition \ref{prop1}$'$,
involving certain types of moderate weights.

\par

\begin{prop}\label{Prop:WeightedGSCharSTFT}
Let $s,\sigma >0$ be such that $s+\sigma \ge 1$,
$\phi \in \Sigma _{s,\sigma}^{\sigma ,s}
(\rr {2d})\setminus 0$  ($\phi \in \maclS
_{s,\sigma}^{\sigma ,s}
(\rr {2d})\setminus 0$), $r>0$, $\omega \in
\mascP _{s,\sigma}(\rr {2d})$ ($\omega \in
\mascP _{s,\sigma}^0(\rr {2d})$) and let
$a$ be a Gelfand-Shilov distribution on $\rr {2d}$.
Then the following is true:
\begin{enumerate}
\item If $a\in C^\infty (\rr {2d})$ and satisfies
\begin{equation}\label{GelfRelCond3}
|\partial _{x}^{\alpha }\partial _{\xi }^{\beta}
a(x,\xi )|\lesssim h ^{|\alpha +\beta|}
\alpha !^{\sigma }\beta !^{s}
\omega (x,\xi ),
\end{equation}
for every $h>0$ (for some $h>0$), then
\begin{equation}\label{GelfRelCond3A}
|V_\phi a(x,\xi ,\eta ,y)|
\lesssim \omega (x,\xi )e^{-r (|\eta |^{\frac 1{\sigma}} +
|y|^{\frac 1s}) },
\end{equation}
for every $r >0$ (for some $r>0$);

\vrum

\item If \eqref{GelfRelCond3A} holds for every $r >0$
(for some $r>0$), then $a\in C^\infty (\rr {2d})$ and
\eqref{GelfRelCond3} holds for every $h>0$ (for some $h>0$).
\end{enumerate}
\end{prop}

\par



\par

\begin{proof}
We shall use similar arguments as in the proof of Proposition
\ref{prop1}.
Let $X=(x,\xi )\in \rr {2d}$, $Y=(y,\eta )\in \rr {2d}$,
$Z=(z,\zeta )\in \rr {2d}$ and let
$\phi  \in \Sigma _{s,\sigma}^{\sigma ,s} (\rr {2d})
\setminus 0$.
Suppose that $\omega \in \mascP _{s,\sigma}(\rr {2d})$
and that \eqref{GelfRelCond3} holds for all $h>0$.
If
$$
F_{X}(Y)\equiv \frac{a_k(Y+X)\overline{\phi (Y)}}
{\omega (X)}
$$
then the fact that $\omega (X)
\lesssim
\omega (Y+X)e^{r_0(|y|^{\frac 1s} +|\eta |^{\frac 1\sigma})}$
gives that
$$
\sets {Y\mapsto F_{X}(Y)}{X\in \rr {2d}}
$$
is a bounded set of $\Sigma _{s,\sigma}^{\sigma,s}$. Hence
$$
|\partial _y^\alpha \partial _\eta^\beta F_{X}(y,\eta)|
\lesssim
h^{|\alpha+\beta |}\alpha !^\sigma \beta!^s
e^{-r(|y|^{\frac 1s}+|\eta|^{\frac 1\sigma})},
$$
for every $h,r >0$. In particular,
\begin{equation}\label{F_kXEsts}
\begin{aligned}
|F_{X}(y,\eta)
&\lesssim
e^{-r(|y|^{\frac 1s}+|\eta|^{\frac 1\sigma})}
\\[1ex]
\text{and}\qquad
|(\mascF  F_{X})(\zeta ,z)|
&\lesssim
e^{-r(|z|^{\frac 1s}+|\zeta |^{\frac 1\sigma})},
\end{aligned}
\end{equation}
for every $r>0$. Since
$$
|V _{\phi}a(x,\xi ,\eta ,y)|
=
|(\mascF  F_{X})(\eta ,y)\omega (X)|,
$$
it follows that \eqref{GelfRelCond3A} holds for
all $r>0$. This gives (1) in the case when
$\omega \in \mascP _{s,\sigma}(\rr {2d})$ and
$\phi \in \Sigma _{s,\sigma}^{\sigma ,s}(\rr {2d})
\setminus 0$. In the same way, (1) follows in the case
when
$\omega \in \mascP _{s,\sigma}^0(\rr {2d})$ and
$\phi \in \maclS _{s,\sigma}^{\sigma ,s}(\rr {2d})
\setminus 0$. The details are left for the reader.

\par

Next we prove (2) in the case $\cdots$. Therefore,
suppose \eqref{GelfRelCond3A} holds for
all $r>0$. Then $a$ is smooth in view of
Proposition \ref{prop1}$'$.

\par

By differentiation, \eqref{Eq:STFTInversionFormula},
the fact that
$$
\omega (Z)\lesssim \omega (X)
e^{r_0(|x-z|^{\frac 1s}+|\xi -\zeta |^{\frac 1\sigma})},
$$
and the fact that $\phi \in
\Sigma _{s,\sigma}^{\sigma ,s}$ we get
\begin{multline*}
|\partial _x^\alpha \partial _\xi ^\beta a(x,\xi )|
\\[1ex]
\lesssim
\underset{\delta \le \beta} {\sum _{\gamma \le \alpha}}
{\alpha \choose \gamma}{\beta \choose \delta}
\iint _{\rr {4d}} |\eta ^\gamma y^\delta
V_\phi a(z,\zeta ,\eta ,y )
(\partial _x^{\alpha -\gamma}\partial _\xi ^{\beta -\delta}
\phi )(X-Z)|
\, dYdZ
\\[1ex]
\lesssim
\underset{\delta \le \beta} {\sum _{\gamma \le \alpha}}
{\alpha \choose \gamma}{\beta \choose \delta}
h^{|\alpha +\beta -\gamma -\delta |}
(\alpha -\gamma )!^\sigma
(\beta -\delta )!^s
I_{\gamma ,\delta}(X),
\end{multline*}
where
\begin{multline*}
I_{\gamma ,\delta}(X)
=
\iint _{\rr {4d}}
\omega (Z)
|\eta ^\gamma y^\delta
e^{-(r+r_0)(|x-z|^{\frac 1s}+|y|^{\frac 1s}
+|\xi -\zeta |^{\frac 1\sigma}+|\eta |^{\frac 1\sigma})}
\, dYdZ
\\[1ex]
\lesssim
\omega (X)
\iint _{\rr {4d}}
|\eta ^\gamma y^\delta
e^{-r(|z|^{\frac 1s}+|y|^{\frac 1s}
+|\zeta |^{\frac 1\sigma}+|\eta |^{\frac 1\sigma})}
\, dYdZ
\\[1ex]
\lesssim
h^{|\gamma +\delta |}\gamma !^\sigma \delta !^s
\omega (X)
\iint _{\rr {4d}}
e^{-\frac r2(|z|^{\frac 1s}+|y|^{\frac 1s}
+|\zeta |^{\frac 1\sigma}+|\eta |^{\frac 1\sigma})}
\, dYdZ
\\[1ex]
\asymp
h^{|\gamma +\delta |}\gamma !^\sigma \delta !^s
\omega (X)
\end{multline*}
for every $h,r>0$.
Here the last inequality follows from
\eqref{Eq:StirlingTypeRel}. It follows that
\eqref{GelfRelCond3} holds for every $h>0$
by using the estimates above and similar computations
as in \eqref{Eq:GSTypeEst1}.

\par

The remaining case follows by similar arguments
and is left for the reader.
\end{proof}




\par

\section{Invariance, continuity and composition properties for
pseudo-differential operators}\label{sec3}

\par

In this section we deduce invariance, continuity and composition properties for
pseudo-differential operators with symbols in the classes considered in the
previous sections. In the first part we show that for any such class $S$, the
set $\op _A (S)$ of pseudo-differential operators is independent of the 
matrix $A$. Thereafter we deduce that such operators are continuous
on Gelfand-Shilov spaces and their duals. In the last part we deduce that these
operator classes are closed under compositions.

\par

\subsection{Invariance properties}

\par

An important ingredient in these considerations concerns mapping properties
for the operator $e^{i\scal {AD_\xi}{D_x}}$. In fact we have the following.

\par

\begin{thm}\label{Thm:CalculiTransf}
Let  $s,s_1,s_2,\sigma ,
\sigma _1,\sigma
_2>0$ be such that
$$
s+\sigma \ge 1,\quad
s_1+\sigma _1\ge 1,\quad
s_2+\sigma _2\ge 1,\quad s_2\le s_1
\quad \text{and}\quad
\sigma _1\le \sigma _2,
$$
and let $A \in \MdR$. Then the following is true:
\begin{enumerate}
\item $e^{i\scal {AD_\xi}{D_x}}$ on $\mascS (\rr {2d})$ restricts to
a homeomorphism on $\maclS _{s_1,\sigma _2}^{\sigma _1,s_2}(\rr {2d})$,
and extends uniquely to a homeomorphism on
$(\maclS _{s_1,\sigma _2}^{\sigma _1,s_2})'(\rr {2d})$;

\vrum

\item if in addition $(s_1,\sigma _1)\neq (\frac 12 ,\frac 12)$ and
$(s_2,\sigma _2)\neq (\frac 12 ,\frac 12)$, then $e^{i\scal {AD_\xi}{D_x}}$
on $\mascS (\rr {2d})$ restricts to a homeomorphism on
$\Sigma _{s_1,\sigma _2}^{\sigma _1,s_2}(\rr {2d})$,
and extends uniquely to a homeomorphism on
$(\Sigma _{s_1,\sigma _2}^{\sigma _1,s_2})'(\rr {2d})$;

\vrum

\item $e^{i\scal {AD_\xi}{D_x}}$ is a homeomorphism
on $\Gamma _{s,\sigma ;0}^{\sigma ,s}(\rr {2d})$;

\vrum

\item if in addition $(s,\sigma )\neq (\frac 12,\frac 12)$,
then $e^{i\scal {AD_\xi}{D_x}}$ is a
homeomorphism on $\Gamma _{s,\sigma}^{\sigma ,s;0}(\rr {2d})$
and on $\Gamma _{s,\sigma ;0}^{\sigma ,s;0}(\rr {2d})$.
%
%
\end{enumerate}
\end{thm}

\par

The assertion (1) in the previous theorem is proved in \cite{CaTo} and is essentially
a special case of Theorem 32 in \cite{Tr}, whereas (2) can be found in
\cite{CaTo,CaWa}. Thus we need to prove Theorem \ref{Thm:CalculiTransf} (3) and (4),
which are extensions of \cite[Theorem 4.6 (3)]{CaTo}.

\par

\begin{proof}
We need to prove (3) and (4) and begin with (3).
Let $\phi \in \maclS _{s,\sigma}^{\sigma, s}(\rr {2d})$ and
$\phi _A = e^{i\scal {AD_\xi }{D_x}}\phi$. Then
$\phi _A\in \maclS _{s,\sigma}^{\sigma, s}(\rr {2d})$, in
view of (1), and
\begin{equation}\label{Eq:stftinv}
    |(V_{\phi _A} (e^{i\scal {AD_\xi}{D_x}}a))(x,\xi ,\eta ,y)|
=
|(V_\phi a)(x-Ay,\xi -A^*\eta ,\eta ,y)|
\end{equation}
by straightforward computations.
Then $a\in\Gamma _{s,\sigma ;0}^{\sigma ,s}(\rr {2d})$ is equivalent
to that for some $h>0$,
$$
|V_\phi a(x,\xi,\eta,y)|
\lesssim
e^{r (|x|^{\frac 1{s}} +|\xi|^{\frac 1{\sigma}})
-h (|\eta|^{\frac 1{\sigma}}
+
|y|^{\frac 1{s}})},
$$
holds for every $r >0$, in view of Proposition \ref{prop2'}.
By \eqref{Eq:stftinv} and (1) it follows by straightforward
computation, that the latter condition is invariant under
the mapping $e^{i\scal {AD_\xi }{D_x}}$, and (3) follows from
these invariance properties.
By similar arguments, taking $\phi \in \Sigma _{s,\sigma}^{\sigma, s}(\rr {2d})$
and using (2) instead of (1), we deduce (4).
\end{proof}

\par

\begin{cor}
Let $s,\sigma >0$ be such that $s+\sigma \ge 1$ and
$\sigma \le s$. Then $e^{i\scal {AD_\xi}{D_x}}$ is
a homeomorphism on $\maclS _{s}^{\sigma}(\rr {2d})$,
$\Sigma _{s}^{\sigma}(\rr {2d})$,
$(\maclS _{s}^{\sigma})'(\rr {2d})$ and on
$(\Sigma _{s}^{\sigma})'(\rr {2d})$.
\end{cor}

\par

We also have the following extension of (3) and (4) in
\cite[Theorem 4.1]{CaTo}.


\par

\begin{thm}\label{Thm:CalculiTransfbis}
Let $\omega \in \mascP _{s,\sigma}(\rr {2d})$,
$s,\sigma>0$ be such that $s+\sigma\ge 1$.
Then $a\in \Gamma _{(\omega)} ^{\sigma,s;0}(\rr {2d})$
if and only if $e^{i\scal 
{AD_\xi}{D_x}}a\in \Gamma _{(\omega)} ^{\sigma,s;0}(\rr {2d})$.
\end{thm}

\par

We need some preparation for the proof and start with the following
proposition.

\par

\begin{prop} \label{prepalg}
Let $s,\sigma >0$ be such that $s+\sigma \ge 1$ and
$(s,\sigma )\neq (\frac 12,\frac 12)$,
$\phi \in \Sigma _{s,\sigma}^{\sigma ,s}(\rr {2d})\setminus 0$,
$\omega \in \mascP _{s,\sigma }(\rr {2d})$ and let $a$ be a
Gelfand-Shilov distribution on $\rr {2d}$.
Then the following conditions are equivalent:
\begin{enumerate}
\item $ a \in \Gamma ^{\sigma ,s;0}_{(\omega)}(\rr {2d})$;

\vrum

\item for every $\alpha ,\beta \in \nn d$, $h>0$, $R>0$ and $x,y,\xi ,\eta$
in $\rr {d}$ it holds
\begin{equation}\label{STFTEst1}
\left |
\partial _x^\alpha \partial _\xi ^\beta
\left ( e^{i(\scal x\eta +\scal y\xi}
V_{\phi}a(x,\xi ,\eta ,y )
\right )
\right |
\lesssim h ^{|\alpha +\beta |}\alpha!^\sigma \beta !^s
\omega (x,\xi )
e^{-R(|y|^{\frac 1s}+|\eta |^{\frac 1\sigma})}\text ;
\end{equation}

\vrum

\item for $\alpha =\beta =0$, \eqref{STFTEst1} holds for every
$h>0$, $R>0$ and $x,y,\xi ,\eta \in \rr {d}$.
\end{enumerate}
\end{prop}

\par

\begin{proof}
Obviously, (2) implies (3). Assume now that (1) holds. Let 
$$
F_a(x,\xi ,y,\eta )=a(x+y,\xi +\eta )\phi(y,\eta ).
$$
By straightforward application of Leibniz rule in combination with
\eqref{estomega} we obtain 
$$
|\partial _x^\alpha \partial _\xi ^\beta F_a(x,\xi ,y,\eta )|
\lesssim
h^{|\alpha + \beta |}\alpha !^\sigma  \beta !^s
\omega(x,\xi )e^{-R(|y|^{\frac 1s}+|\eta |^{\frac 1\sigma})}
$$
for every $h >0$ and $ R>0$.
Hence, if 
$$
G_{a, h, x,\xi }(y,\eta )= \frac{\partial _x^\alpha \partial _\xi ^\beta F_a(x,\xi ,y,\eta )}
{h^{|\alpha +\beta |}\alpha !^\sigma \beta !^s\omega(x,\xi )},
$$
then $\{ G_{a, h, x,\xi } \}_{x,\xi  \in \mathbf{R}^{d}}$ is a bounded set in
$\Sigma _{s,\sigma }^{\sigma ,s}(\mathbf{R}^{2d})$ for every fixed $h >0$. If
$\mascF _2 F_a$
is the partial Fourier transform of $F_a(x,\xi ,y,\eta )$ with respect to the $(y,\eta )$-variable,
we get
$$
|\partial _x^\alpha \partial _\xi ^\beta (\mascF _2 F_a)(x,\xi , \zeta, z)|
\lesssim
h^{|\alpha +\beta |}\alpha!^\sigma \beta !^s
\omega(x,\xi )e^{-R(|z|^{\frac 1s} + |\zeta|^{\frac 1\sigma })},
$$
for every $h>0$ and $R>0$. This is the same as (2).

\par

It remains to prove that (3) implies (1), but this follows by similar arguments as in the
proof of Proposition \ref{prop1}.
The details are left for the reader.
\end{proof}

\par

\begin{prop}\label{SymbClassModSpace}
Let $R>0$, $q\in [1,\infty ]$, $s,\sigma >0$ be such that $s+\sigma \ge 1$
and $(s,\sigma )\neq (\frac 12,\frac 12)$, $\phi \in \Sigma _{s,\sigma}^{\sigma ,s}
(\rr {2d})\setminus 0$,
$\omega \in \mascP _{s,\sigma}(\mathbf{R}^{2d}),$ and let
$$
\omega _R(x,\xi, \eta, y) = \omega(x,\xi) e^{-R(|y|^{\frac 1s} + |\eta|^{\frac 1\sigma})}.
$$
Then
\begin{equation}\label{iden}
\Gamma^{\sigma ,s;0}_{(\omega)}(\rr {2d})
=
\bigcap _{R>0}\sets {a\in (\Sigma _{s,\sigma}^{\sigma ,s})'
(\rr {2d})}{\nm {\omega
_R^{-1}V_\phi a}{L^{\infty ,q}} <\infty }.
\end{equation}
\end{prop}

\par

\begin{proof}
Let $\phi _0\in \Sigma _{s,\sigma}^{\sigma ,s}(\rr {2d})\setminus 0$, 
$a \in (\Sigma _{s,\sigma}^{\sigma ,s})'(\mathbf{R}^{2d})$, and set
\begin{gather*}
F_{0,a}(X,Y)=|(V_{\phi _0}a)(x,\xi, \eta, y)|,
\quad
F_{a}(X,Y)=|(V_{\phi}a)(x,\xi, \eta, y)|
\\[1ex]
\text{and}\quad 
G(x,\xi, \eta,y)=|(V_\phi {\phi _0})(x,\xi, \eta, y)|,
\end{gather*}
where $X=(x,\xi)$ and 
$Y=(y,\eta)$. Since $V_\phi {\phi _0}\in \Sigma _{s,\sigma}^{\sigma ,s}(\rr {4d})$, we have
\begin{equation}\label{GEst1}
0\le G(x,\xi, \eta,y) \lesssim e^{-R(|x|^{\frac 1s}+|\xi|^{\frac 1\sigma}+|\eta|^{\frac 1\sigma}+|y|^{\frac 1s})}
\quad \text{for every}\quad
R>0 .
\end{equation}

\par

By \cite[Lemma 11.3.3]{Gro}, we have
$F_a \lesssim F_{0,a} \ast G$.
We obtain 
\begin{multline}\label{pippo}
(\omega_R^{-1} \cdot F_a)(X,Y)
\\
\lesssim \omega(X)^{-1}e^{R(|y|^{\frac 1s}+|\eta|^{\frac 1\sigma})}
\iint_{\rr {4d}} F_{0,a}(X-X_1, Y-Y_1)G(X_1, Y_1)\, dX_1 dY_1 
\\
\lesssim \iint_{\rr {4d}} (\omega^{-1}_{cR} \cdot F_{0,a})(X-X_1, Y-Y_1)G_1(X_1,Y_1)\,  dX_1 dY_1
\end{multline}
for some $G_1$ satisfying \eqref{GEst1} in place of  $G$ and some $c>0$ independent of $R$.
By applying the $L^\infty$-norm on the last inequality we get
\begin{multline*}
\nm {\omega _R^{-1}F_a}{L^{\infty}(\rr {4d})}
\\[1ex]
\lesssim \sup_{Y}\left(\iint_{\rr {4d}} 
\big ( \sup (\omega^{-1}_{cR} \cdot F_{0,a})(\cdot , Y-Y_1)\big )
G_1(X_1, Y_1)\,  dX_1 dY_1\right)
\\[1ex]
\le
\sup _Y \left( \nm {(\omega^{-1}_{cR} \cdot F_{0,a})(\cdot -(0,Y)}{L^{\infty ,q}}\right)
\nm {G_1}{L^{1,q'}}
\asymp \nm {\omega^{-1}_{cR} \cdot F_{0,a}}{L^{\infty ,q}}.
\end{multline*}

\par

We only consider the case $q<\infty$ when proving the opposite inequality.
The case $q=\infty$ follows by similar arguments and is left for the reader.

\par

By \eqref{pippo} we have 
\[
\| \omega^{-1}_{R} \cdot F_{a} \| _{L^{\infty,q}}^q
\lesssim  \int_{\rr {2d}} \left(\sup H(\cdo,Y)^q  \right)\,dY,
\]
where $H=K_1*G$ and $K_j=\omega^{-1}_{jcR} \cdot F_{0,a},\,j\ge1.$

By Minkowski's inequality, letting $Y_1=(y_1,\eta_1)$ as variables of integration, we get
\begin{multline*}
\sup_X H(X,Y)
\\[1ex]
\lesssim  \iint_{\rr {4d}} \big (\sup K_2(\cdo, Y-Y_1)\big )
e^{-cR(|y-y_1|^{\frac 1s}+|\eta-\eta_1|^{\frac 1\sigma})}
G(X_1, Y_1)\, dX_1 dY_1 
\\[1ex]
\lesssim \| K_2 \|_{L^\infty}
\iint_{\rr {4d}}  e^{-cR(|y-y_1|^{\frac 1s}+|\eta-\eta_1|^{\frac 1\sigma})}
G(X_1, Y_1)\, dX_1 dY_1.
\end{multline*}
By combining these estimates we get
\begin{multline*}
    \| \omega^{-1}_{R} \cdot F_{a} \|_{L^{\infty,q}}^q
    \\[1ex]
\lesssim \|K_2 \|_{L^\infty}^q\int_{\rr {2d}}\left(\iint_{\rr {4d}} 
e^{-cR(|y-y_1|^{\frac 1s}+|\eta-\eta_1|^{\frac 1\sigma})}
G(X_1, Y_1) \, dX_1dY_1 \right)\, dY
 \\[1ex]
 \asymp \|K_2 \|_{L^\infty}^q.
\end{multline*}
That is,

$$
\| \omega^{-1}_{R} \cdot F_{a} \| _{L^{\infty,q}} 
\lesssim 
\| \omega^{-1}_{2cR} \cdot F_{0,a} \|_{L^\infty},
$$
and the result follows.
\end{proof}

\par


\par

\begin{proof}[Proof of Theorem \ref{Thm:CalculiTransfbis}]
The case $s=\sigma =\frac 12$ follows from \cite[Theorem 4.1]{CaTo}.
We may therefore assume that $(s,\sigma)\neq(\frac 12, \frac 12)$.
Let $\phi \in \Sigma _{s,\sigma}^{\sigma, s}(\rr {2d})$ and
$\phi _A = e^{i\scal {AD_\xi }{D_x}}\phi$. Then
$\phi _A\in \Sigma  _{s,\sigma}^{\sigma, s}(\rr {2d})$, in view
of (2) in Theorem \ref{Thm:CalculiTransf}.

\par

Also let
$$
\omega _{A,R}(x,\xi ,\eta ,y) = \omega (x-Ay,\xi -A^*\eta )
e^{-R(|y|^{\frac 1s}+|\eta |^{\frac 1\sigma})}.
$$
By straight-forward applications of Parseval's formula, we get
$$
|(V_{\phi _A} (e^{i\scal {AD_\xi}{D_x}}a))(x,\xi ,\eta ,y)|
=
|(V_\phi a)(x-Ay,\xi -A^*\eta ,\eta ,y)|
$$
(cf. Proposition 1.7 in \cite{To7} and its proof). This gives
$$
\nm {\omega _{0,R}^{-1}V_{\phi} a}{L^{p,q}} =
\nm {\omega _{A,R}^{-1}V_{\phi _A} (e^{i\scal {AD_\xi }{D_x}}a)}{L^{p,q}}.
$$
Hence Proposition \ref{SymbClassModSpace} gives
\begin{alignat*}{1}
a\in\Gamma _{(\omega)} ^{\sigma,s;0}(\rr {2d})
\quad &\Leftrightarrow \quad
\nm {\omega _{0,R}^{-1}V_\phi a}{L^\infty}<\infty \quad
\text{for every}
\ R>0
\\[1ex]
&\Leftrightarrow \quad
\nm {\omega _{t,R}^{-1}V_{\phi _A} (e^{i\scal
{AD_\xi}{D_x}}a)}{L^\infty} <\infty \quad
\text{for every} \ R>0
\\[1ex]
&\Leftrightarrow \quad
\nm {\omega _{0,R}^{-1}V_{\phi _A} (e^{i\scal {AD_\xi}{D_x}}a)}{L^\infty}
<\infty \quad
\text{for every} \ R>0
\\[1ex]
&\Leftrightarrow \quad
e^{i\scal {AD_\xi}{D_x}}a \in
\Gamma _{(\omega)} ^{\sigma,s;0}(\rr {2d}), 
\end{alignat*}
and the result follows in this case. Here the third equivalence follows from the fact that
$$
\omega _{0,R+c}\lesssim \omega _{t,R}\lesssim \omega _{0,R-c},
$$
for some $c>0$.

\par

\end{proof}

\par

\par 

We note that if $A,B \in \MdR$ and $a,b\in 
(\maclS _{s,\sigma}^{\sigma ,s})'(\rr {2d})$ or $a,b\in (\Sigma
_{s,\sigma}^{\sigma ,s})'(\rr {2d})$, then
the first part of the previous proof shows that
\begin{equation}\label{EqCalculiTransfer}
\op _A(a) = \op _B(b)\quad \Leftrightarrow \quad 
e^{i\scal {AD_\xi}{D_x}}a = e^{i\scal {BD_\xi}{D_x}}b .
\end{equation}
%
%

\par

The following result follows from Theorems \ref{Thm:CalculiTransf}
and \ref{Thm:CalculiTransfbis}. The details are left for the reader.

\par

\begin{thm}\label{ThmCalculiTransf2}
Let $s,s_1,s_2,\sigma ,\sigma _1,\sigma _2>0$ be such that
$$
s+\sigma \ge 1,\quad
s_1+\sigma _1\ge 1,\quad
s_2+\sigma _2\ge 1,\quad s_2\le s_1
\quad \text{and}\quad
\sigma _1\le \sigma _2,
$$
$A, B \in \MdR$, $\omega \in \mascP _{s,\sigma}(\rr {2d})$,
and let $a$ and $b$
be Gelfand-Shilov distributions such that $\op _A(a)=\op _B(b)$.
Then the following is true:
\begin{enumerate}
\item $a\in \maclS ^{\sigma _1,s_2}_{s_1,\sigma_2}(\rr {2d})$
if and only if $b\in
\maclS ^{\sigma _1,s_2}_{s_1,\sigma _2}(\rr {2d})$, and
$a\in (\maclS ^{\sigma _1,s_2}_{s_1,\sigma _2})'(\rr {2d})$
if and only if $b\in (\maclS ^{\sigma _1,s_2}_{s_1,\sigma _2})'
(\rr {2d})$;

\vrum

\item $a\in \Sigma
^{\sigma _1,s_2}_{s_1,\sigma _2}(\rr {2d})$ if and only if $b\in
\Sigma ^{\sigma _1,s_2}_{s_1,\sigma _2}(\rr {2d})$,
and $a\in (\Sigma ^{\sigma _1,s_2}_{s_1,\sigma _2})'(\rr {2d})$
if and only if $b\in (\Sigma ^{\sigma _1,s_2}
_{s_1,\sigma _2})'(\rr {2d})$;

\vrum

\item $a\in \Gamma _{s,\sigma ;0}^{\sigma ,s}(\rr {2d})$ if and only if
$b\in \Gamma _{s,\sigma ;0}^{\sigma ,s}(\rr {2d})$. If in addition
$(s,\sigma )\neq (\frac 12,\frac 12)$, then 
$a\in \Gamma _{s,\sigma}^{\sigma ,s;0}(\rr {2d})$ if and only if $b\in
\Gamma _{s,\sigma}^{\sigma ,s;0}(\rr {2d})$, and
$a\in \Gamma _{s,\sigma ;0}^{\sigma ,s;0}(\rr {2d})$ if and only if $b\in
\Gamma _{s,\sigma ;0}^{\sigma ,s;0}(\rr {2d})$;

\vrum

\item $a\in \Gamma _{(\omega)} ^{\sigma ,s;0}(\rr {2d})$ if and only
if $b\in \Gamma _{(\omega)} ^{\sigma ,s;0}(\rr {2d})$.
\end{enumerate}
\end{thm}

\par

\subsection{Continuity for pseudo-differential operators with symbols
of infinite order on Gelfand-Shilov spaces of functions and
distributions}

\par

Next we deduce continuity for pseudo-differential operators with symbols
in the classes in Definitions \ref{Def:GammaSymb2}
and \ref{Def:ExtGSSymbClasses}. We begin with the case when
the symbols belong to $\Gamma _{s,\sigma ;0}^{\sigma ,s}(\rr {2d})$.

\par

\begin{thm}\label{Thm:theorem2}
Let $A \in \MdR$, $s,\sigma >0$ be such that $s+\sigma \ge 1$, and let
$a\in \Gamma _{s,\sigma ;0}^{\sigma ,s}(\rr {2d})$. Then $\op_A (a)$
is continuous on $\maclS _s^\sigma (\rr d)$ and on
$(\maclS _s^\sigma )'(\rr d)$.
\end{thm}

\par

For the proof we need the following result.

\par

\begin{lemma}\label{Lemma:theorem2}
Let $s,\sigma >0$ be such that $s+\sigma \ge 1$,
$h_1>0$, $\Omega _1$ be a bounded set in
$\maclS_{s;h_1}^\sigma (\rr d)$, and let
$$
h_2\ge 2^{1+s}h_1
\quad \text{and}\quad
h_3\ge 2^{2+2s}h_1.
$$
Then
\begin{align*}
\Omega _2 &= \Sets {x\mapsto \frac {x^\gamma f(x)}
{(2^{1+s}h_1)^{|\gamma |} \gamma !^s}}
{f\in \Omega _1,\ \gamma \in \nn d}
\intertext{is a bounded set in $S_{s;h_2}^\sigma (\rr d)$,
and}
\Omega _3 &= \Sets {x\mapsto \frac {D^\delta x^\gamma f(x)}
{(2^{2+2s}h_1)^{|\gamma +\delta |} \gamma !^s\delta !^\sigma}}
{f\in \Omega _1,\ \gamma ,\delta \in \nn d}.
\end{align*}
is a bounded sets in $S_{s;h_3}^\sigma (\rr d)$.
\end{lemma}

\par

\begin{proof}
Since $\Omega _1$ is a bounded set in $S_{s;h_1}^\sigma
(\rr d)$, there are constants $C>0$ such that
\begin{equation}\label{Eq:GSCond}
|x^\alpha D^\beta f(x)| \le Ch_1^{|\alpha +\beta |}
\alpha !^s\beta !^\sigma ,\quad \alpha ,\beta \in \nn d,
\end{equation}
for every $f\in \Omega _1$. We shall prove
that \eqref{Eq:GSCond} is true for all $f\in \Omega _2$
for a new choice of $C>0$, and $h_2$ in place of $h_1$.

\par

Let
$f\in \Omega _2$. Then
$$
f=\frac {x^\gamma f_0(x)}
{(2^{1+s}h_1)^{|\gamma |} \gamma !^s}
$$
for some $f_0\in \Omega _1$ and $\gamma \in \nn d$.
Then
\begin{multline*}
|x^\alpha D^\beta f(x)|
=
\left |
\frac {x^\alpha D^\beta (x^\gamma f_0)(x)}
{(2^{1+s}h_1)^{|\gamma |} \gamma !^s}
\right |
\\[1ex]
\le
\sum _{\gamma _0\le \gamma ,\beta}
{{\beta} \choose {\gamma _0}}
\frac {\gamma !}{(\gamma -\gamma _0)!}
\cdot
\frac {|x^{\alpha +\gamma -\gamma _0}
\partial ^{\beta -\gamma _0}f_0(x)|}
{(2^{1+s}h_1)^{|\gamma |}\gamma !^s}
\\[1ex]
\lesssim
\sum _{\gamma _0\le \gamma ,\beta}
{{\beta} \choose {\gamma _0}}
{{\gamma} \choose {\gamma _0}}\gamma _0!
\cdot
\frac {h_1^{|\alpha +\beta +\gamma -2\gamma _0|}
(\alpha +\gamma -\gamma _0)!^s(\beta -\gamma _0)!^\sigma}
{(2^{1+s}h_1)^{|\gamma |}\gamma !^s}
\\[1ex]
\lesssim
h_1^{|\alpha +\beta |}\alpha !^s\beta !^\sigma
\sum _{\gamma _0\le \gamma ,\beta}
{{\beta} \choose {\gamma _0}}
{{\gamma} \choose {\gamma _0}}
2^{-(1+s)|\gamma |}
\left (
\frac {(\alpha +\gamma -\gamma _0)!\gamma _0!}
{\alpha !\gamma !}
\right )^s
\left (
\frac {(\beta -\gamma _0)!\gamma _0!}
{\beta !}
\right )^\sigma
\\[1ex]
\lesssim
h_1^{|\alpha +\beta |}\alpha !^s\beta !^\sigma
\sum _{\gamma _0\le \gamma ,\beta}
{{\beta} \choose {\gamma _0}}
{{\gamma} \choose {\gamma _0}}
2^{-(1+s)|\gamma |}
\left (
\frac {(\alpha +\gamma -\gamma _0)!\gamma _0!}
{(\alpha +\gamma )!}
\right )^s
{{\alpha +\gamma } \choose {\gamma}}^s
\\[1ex]
\lesssim
h_1^{|\alpha +\beta |}\alpha !^s\beta !^\sigma
\sum _{\gamma _0\le \gamma ,\beta}
2^{|\beta |}2^{|\gamma|}2^{-(1+s)|\gamma |}\cdot 1\cdot
2^{s|\alpha +\gamma |}
\\[1ex]
\lesssim
2^{s|\alpha |}2^{|\beta |}h_1^{|\alpha +\beta |}
\alpha !^s\beta !^\sigma \sum _{\gamma _0\le \beta} 1.
\end{multline*}
Since
$$
\sum _{\gamma _0\le \beta} 1 \lesssim 2^{s|\beta|},
$$
we get
$$
|x^\alpha D^\beta f(x)|
\le
C2^{s|\alpha |}2^{(1+s)|\beta |}h_1^{|\alpha +\beta |}
\alpha !^s\beta !^\sigma
\le
Ch_2^{|\alpha +\beta|} \alpha !^s\beta !^\sigma
$$
for some constant $C$ which is independent of $f$, and the
assertion on $\Omega _2$ follows.

\par

The same type of arguments shows that
\begin{equation}\label{Eq:SetBounded}
\Sets {x\mapsto \frac {D^\delta f(x)}
{(2^{1+s}h_1)^{|\delta |}\delta !^\sigma}}
{f\in \Omega _1,\ \delta \in \nn d}
\end{equation}
is a bounded set in $\maclS_{s;h_2}^\sigma (\rr d)$,
and the boundedness of $\Omega _3$ in $\maclS_{s;h_3}^\sigma
(\rr d)$ follows by combining the
boundedness of $\Omega _2$ and the
boundedness of \eqref{Eq:SetBounded}
in $\maclS_{s;h_2}^\sigma (\rr d)$.
\end{proof}

\par

\begin{lemma}\label{Lemma:mstEst}
Let $s,\tau >0$, and set
$$
m_s (t)=
\sum_{j=0}^{\infty}\frac{t^j}{(j!)^{2s}}
\quad \text{and}\quad
m_{s , \tau}(x)= m_s (\tau \px^2)
\quad t \geq 0,\ x \in \rr d
$$
Then
\begin{equation} \label{paola}
C^{-1}e^{(2s -\ep) \tau^{\frac{1}{2s }}
\px ^{\frac{1}{s}}}
\leq
m_{s , \tau}(x)
\leq
C e^{(2s +\ep)  \tau^{\frac{1}{2s}}\px^{\frac{1}{s}}},
\end{equation}
for every $\ep>0$, and
\begin{equation}\label{Eq:mstEst}
\frac {x^\alpha}{m_{s,\tau}(x)} \lesssim h_0^{|\alpha |}
\alpha !^s e^{-r|x|^{\frac 1s}},
\end{equation}
for some positive constant $r$ which depends on
$d$, $s$ and $\tau$ only.
\end{lemma}

\par

The estimate \eqref{paola} follows from \cite{IV}, and
\eqref{Eq:mstEst} also follows from computations given in
e.{\,}g. \cite{IV,CaTo}. For sake of completeness we present
a proof of \eqref{Eq:mstEst}.

\par

\begin{proof}
We have
$$
\frac {x^\alpha}{m_{s,\tau}(x)}
\lesssim \prod _{j=1}^d
g_{\alpha _j}(x_j),
$$
where
$$
g_k(t) = t^ke^{-2r_0t^{\frac 1s}},\qquad t\ge 0,
$$
for some $r_0>0$ depending only on $d$, $s$ and $\tau$. Let
$$
g_{0,k}(t) = C_k e^{-r_0t},\qquad t\ge 0,
$$
where
$$
C_k = \sup _{t\ge 0}(t^{sk}e^{-r_0t}).
$$
Then $g_k(t) \le g_{0,k}(t^{\frac 1s})$, and the result follows
if we prove $C_k \lesssim h_0^kk!^s$.

\par

By straight-forward computations, it follows that the maximum
of $t^{sk}e^{-r_0t}$ is attained at $t=sk/r_0$, giving that
$$
C_k = \left ( \frac {s}{r_0e} \right )^{sk}(k^k)^s
\lesssim
h_0^kk!^s,\qquad h_0=\left ( \frac s{r_0}\right )^s,
$$
where the last inequality follows from Stirling's formula.
This gives the result.
\end{proof}

\par

\begin{proof}[Proof of Theorem \ref{Thm:theorem2}]
By Theorem \ref{Thm:CalculiTransf} it suffices to consider the case $A=0$,
that is for the operator 
$$
\op _0(a)f(x)= (2\pi )^{-\frac d2} \int_{\rr d} 
a(x,\xi) \widehat{f}(\xi)e^{i\langle x,\xi \rangle} \, d\xi.
$$

\par

Observe that
$$
\frac{1}{m_{s , \tau}(x)} \sum_{j=0}^{\infty} \frac{\tau ^j}
{(j!)^{2s }}(1-\Delta_{\xi})^j
e^{i\langle x,\xi \rangle} = e^{i\langle x,\xi \rangle}.
$$
Let now $h_1>0$ and $f \in \Omega,$ where $\Omega$ is a bounded subset of $\maclS_{s,h_1}^\sigma (\rr d)$.
For fixed $\alpha, \beta \in \nn d$ we get
\begin{multline}\label{Eq:Op(a)Comp}
(2\pi )^{\frac d2} x^{\alpha}D_x^{\beta} (\textrm{Op}_0(a)f)(x)
\\[1ex]
= x^{\alpha} \sum
_{\gamma \leq \beta} {\beta \choose {\gamma} } \int_{\rr d}
 \xi^{\gamma} D_x^{\beta-\gamma}a(x,\xi)
 \widehat{f}(\xi ) e^{i \langle x,\xi \rangle} \,  d\xi 
\\[1ex]
= \frac{x^{\alpha}}{m_{s ,\tau}(x)} \sum_{\gamma \leq \beta}
\binom{\beta}{\gamma} 
g_{\tau ,\beta ,\gamma}(x),
\end{multline}
$$
g_{\tau ,\beta ,\gamma}(x)
=
\sum_{j=0}^{\infty} \frac{\tau ^j}{(j!)^{2s}}
\int_{\rr d}  (1-\Delta_{\xi} )^j
\left (  \xi^{\gamma} D_x^{\beta-\gamma}a(x,\xi)
\widehat{f}(\xi ) \right )  e^{i \langle x,\xi \rangle}
\, d\xi.
$$

\par

By Lemma \ref{Lemma:theorem2} and the fact that $(2j)!\le 2^jj!^2$,
it follows that for some
$h>0$,
$$
\Omega
=
\Sets {\xi \mapsto \frac {(1-\Delta_{\xi} )^j
(  \xi^{\gamma} D_x^{\beta}a(x,\xi)
\widehat{f}(\xi ))}{h^{|\beta +\gamma |+j}j!^{2s}\gamma !^\sigma
\beta !^\sigma e^{r|x|^{\frac 1s}}}}
{j\ge 0,\ \beta ,\gamma \in \nn d }
$$
is bounded in $\maclS _\sigma ^s(\rr d)$ for every $r>0$.
This implies that for some positive constants $h$ and $r_0$
we get
$$
|(1-\Delta_{\xi} )^j (\xi^{\gamma} D_x^{\beta}a(x,\xi)
\widehat{f}(\xi ))|
\lesssim
h^{|\beta +\gamma |+j}j!^{2s}\gamma !^\sigma
\beta !^\sigma e^{r|x|^{\frac 1s}-r_0|\xi |^{\frac 1\sigma}},
$$
for every $r>0$. Hence,
\begin{multline*}
|g_{\tau ,\beta ,\gamma}(x)|
\lesssim
\sum_{j=0}^{\infty}
\frac{\tau ^j}{(j!)^{2s}}
h^{|\beta |+j}j!^{2s}\gamma !^\sigma
(\beta -\gamma)!^\sigma e^{r|x|^{\frac 1s}}
\int_{\rr d} e^{-r_0|\xi |^{\frac 1\sigma}}\, d\xi
\\[1ex]
\lesssim
h^{|\beta |}\beta !^\sigma e^{r|x|^{\frac 1s}}
\sum_{j=0}^{\infty} (\tau h)^j
\asymp
h^{|\beta |}\beta !^\sigma e^{r|x|^{\frac 1s}}
\end{multline*}
for every $r>0$, provided $\tau$ is chosen such that
$\tau h<1$.

\par

By inserting this into \eqref{Eq:Op(a)Comp} and using
Lemma \ref{Lemma:mstEst} we get for some $h>0$ and some $r_0>0$
that
\begin{multline*}
|x^{\alpha}D_x^{\beta} (\textrm{Op}_0(a)f)(x)|
\lesssim
h^{|\alpha |}
\alpha !^s e^{-r_0|x|^{\frac 1s}}
\sum_{\gamma \leq \beta}
\binom{\beta}{\gamma}
|g_{\tau ,\beta ,\gamma}(x)|
\\[1ex]
\lesssim
h^{|\alpha +\beta |}
\alpha !^s\beta !^\sigma e^{-(r_0-r)|x|^{\frac 1s}}
\left (\sum_{\gamma \leq \beta} 1\right )
\lesssim
(2h)^{|\alpha +\beta |}
\alpha !^s\beta !^\sigma ,
\end{multline*}
provided that $r$ above is chosen to be smaller than $r_0$.
Then the continuity of $\op _A(a)$ on $\maclS _{s}^\sigma (\rr d)$
follows. The continuity of $\op _A(a)$ on
$(\maclS _{s}^\sigma )'(\rr d)$ now follows from the preceding
continuity and duality.
\end{proof}

\par

Next we shall discuss corresponding continuity in the
Beurling case. The main idea is to deduce such properties
by suitable estimates on short-time Fourier transforms
of involved functions and distributions.
First we
have the following relation between the
short-time Fourier transforms of the
symbols and kernels of a pseudo-differential operator.

\par

\begin{lemma}\label{Lemma:STFTSymbolKernel}
Let $A \in \MdR$, $s,\sigma >0$ be such that $s+\sigma \ge 1$,
$a\in (\maclS _{s,\sigma}^{\sigma ,s})'(\rr {2d})$ 
($a\in (\Sigma _{s,\sigma}^{\sigma ,s})'(\rr {2d})$),
$\phi \in \maclS _{s,\sigma}^{\sigma ,s}(\rr {2d})$
($\phi \in \Sigma _{s,\sigma}^{\sigma ,s}(\rr {2d})$),
and let
\begin{align*}
K_{A,a}(x,y) &= (2\pi )^{-\frac d2}(\mascF _2^{-1}a)(x-A(x-y),x-y)
\intertext{and}
\psi (x,y) &= (2\pi )^{-\frac d2}(\mascF _2^{-1}\phi )(x-A(x-y),x-y)
\end{align*}
be the kernels of $\op _A(a)$ and $\op _A(\phi )$,
respectively. Then
\begin{multline}\label{Eq:STFTSymbolKernel}
(V_\psi K_{a,A})(x,y,\xi ,\eta) 
\\[1ex]
=
(2\pi )^{-d}e^{i\scal {x-y}{\eta -A^*(\xi +\eta )}}
(V_\phi a)(x-A(x-y),-\eta +A^*(\xi +\eta ),\xi +\eta ,y-x).
\end{multline}
\end{lemma}

\par

The essential parts of \eqref{Eq:STFTSymbolKernel}
is presented in the proof of \cite[Proposition 2.5]{To16}.
In order to be self-contained we here present a short proof.

\par

\begin{proof}
Let
$$
T_A(x,y)=x-A(x-y)
$$
and
$$
Q=Q(x,x_1,y,\xi ,\xi _1,\eta )
=
\scal {x-y}{\xi _1-T_{A^*}(-\eta ,\xi )}
-\scal {x_1}{\xi +\eta}.
$$
By formal computations, using Fourier's
inversion formula we get
\begin{multline*}
(V_\psi K_{a,A})(x,y,\xi ,\eta)
\\[1ex]
=
(2\pi )^{-3d}\iint 
K_{a,A}(x_1,y_1) \overline{\psi (x_1-x,y_1-y)}
e^{-i(\scal {x_1}\xi +\scal {y_1}\eta)}\, dx_1dy_1
\\[1ex]
=
(2\pi )^{-2d}\iint a(x_1,\xi _1)
\overline{\phi (x_1-T_A(x,y),\xi _1-T_{A^*}(-\eta ,\xi ))}
e^{iQ(x,x_1,y,\xi ,\xi _1,\eta )}\, dx_1d\xi _1
\\[1ex]
=
(2\pi )^{-d}e^{i\scal {x-y}{\eta -A^*(\xi +\eta )}}
(V_\phi a)(T_A(x,y),T_{A^*}(-\eta ,\xi ),\xi +\eta ,y-x),
\end{multline*}
where all integrals should be interpreted as suitable
Fourier transforms. This gives the result.
\end{proof}

\par

Before continuing discussing continuity of pseudo-differential
operators, we observe that the previous lemma in combination
with Propositions \ref{prop2} and \ref{prop2'} give the following.

\par

\begin{prop}\label{Prop:prop2B}
Let $A \in \MdR$, $s,\sigma >0$ be such that $s+\sigma \ge 1$ and
$(s,\sigma )\neq (\frac 12,\frac 12)$,
$\phi \in \Sigma _{s}^{\sigma}
(\rr {2d})\setminus 0$,
$a$ be a Gelfand-Shilov distribution on $\rr {2d}$
and let $K_{a,A}$ be the kernel of $\op _A(a)$. 
Then the following conditions are equivalent:
\begin{enumerate}
\item $a\in \Gamma _{s,\sigma}^{\sigma ,s}
(\rr {2d})$ (resp. $a\in \Gamma _{s,\sigma}^{\sigma ,s;0}
(\rr {2d})$);

\vrum

\item for some $r>0$,
$$
|V_\phi K_{a,A}(x,y,\xi ,\eta )|
\lesssim
e^{r (|x-A(x-y)|^{\frac 1{s}} +|\eta -A^*(\xi +\eta )|^{\frac 1{\sigma}})
-h (|\xi +\eta |^{\frac 1{\sigma }}
+
|x-y|^{\frac 1{s}})}
$$
holds for some $h>0$ (for every $h>0$).
\end{enumerate}
\end{prop}

\par

By similar arguments we also get the following.
The details are left for the reader.

\par

\begin{prop}\label{Prop:prop2'B}
Let $A \in \MdR$, $s,\sigma >0$ be such that $s+\sigma \ge 1$,
$\phi \in \maclS _{s}^{\sigma}
(\rr {2d})\setminus 0$,
$a$ be a Gelfand-Shilov distribution on $\rr {2d}$
and let $K_{a,A}$ be the kernel of $\op _A(a)$. 
Then the following conditions are equivalent:
\begin{enumerate}
\item $a\in \Gamma _{s,\sigma ;0}^{\sigma ,s}
(\rr {2d})$ (resp. $a\in \Gamma _{s,\sigma ;0}^{\sigma ,s;0}
(\rr {2d})$);

\vrum

\item for some $h>0$ (for every $h>0$),
$$
|V_\phi K_{a,A}(x,y,\xi ,\eta )|
\lesssim
e^{r (|x-A(x-y)|^{\frac 1{s}} +|\eta -A^*(\xi +\eta )|^{\frac 1{\sigma}})
-h (|\xi +\eta |^{\frac 1{\sigma }}
+
|x-y|^{\frac 1{s}})}
$$
holds for every $r>0$.
\end{enumerate}


\end{prop}

\par

\begin{thm}\label{Thm:theorem1}
Let $A \in \MdR$, $s,\sigma >0$ be such that $s+\sigma \ge 1$
and $(s,\sigma )\neq (\frac 12,\frac 12)$, and let
$a\in \Gamma_{s,\sigma}^{\sigma ,s;0} (\rr {2d})$. Then $\op _A(a)$
is continuous on
$\Sigma _s^\sigma (\rr d)$, and is uniquely extendable
to a continuous map on $(\Sigma _s^\sigma )'(\rr d)$.
\end{thm}

\par

\begin{proof}
By Theorem \ref{Thm:CalculiTransf} we may assume that $A=0$.
Let
$$
g(x) = \op _0(a)f(x) = ({K_{0,a}(x,\cdo )},{\overline f})
=
({h_{a,x}},\overline f),
$$
where $h_{a,x}=K_{0,a}(x,\cdo )$, and let $\phi _j
\in \Sigma _s^\sigma (\rr d)$ be such that
$\nm {\phi _j}{L^2}=1$, $j=1,2$.
By Moyal's identity we get
$$
g(x) = ( {h_{a,x}},\overline f )_{L^2(\rr d)}
=
({V_{\phi _1}h_{a,x}},{V_{\phi _1}\overline f})_{L^2(\rr {2d})},
$$
and applying the short-time Fourier transform on $g$
and using Fubbini's theorem on distributions
we get
$$
V_{\phi _2}g(x,\xi ) = \scal {J(x,\xi ,\cdo )}{F},
$$
where
$$
F(y,\eta ) = V_{\phi _1}f(y,-\eta ),\quad
J(x,\xi ,y,\eta ) = V_\phi K_{0,a}(x,y,\xi ,\eta)
$$
and 
$\phi =\phi _2\otimes \phi _1 .$

\par

Now suppose that $r>0$ is arbitrarily chosen.
By Proposition \ref{prop2} we get for some
$c\in (0,1)$ which depends on $s$ and $\sigma$ only,
that for some $r_0>0$ and with $r_1= (r+r_0)/c$ that
\begin{multline*}
|J(x,\xi ,y,\eta )|
\lesssim
e^{r_0(|x|^{\frac 1s} + |\eta |^{\frac 1\sigma })}
e^{-r_1(|y-x|^{\frac 1s}+|\xi +\eta |^{\frac 1{\sigma}})}
\\[1ex]
\lesssim
e^{-( (cr_1-r_0)|x|^{\frac 1s}+cr_1|\xi |^{\frac 1\sigma})}
e^{r_1|y|^{\frac 1s}+(r_1+r_0)|\eta |^{\frac 1\sigma}}
\\[1ex]
\lesssim
e^{- r(|x|^{\frac 1s}+|\xi |^{\frac 1\sigma})}
e^{r_2(|y|^{\frac 1s}+|\eta |^{\frac 1\sigma})},
\end{multline*}
where $r_2$ only depends on $r$ and $r_0$.

\par

Since $f\in \Sigma _s^\sigma (\rr d)$ we have
$$
|F(x,\xi )|\lesssim \nm f{S_{s;h}^\sigma}
e^{-(1+r_2)(|x|^{\frac 1s}+|\xi |^{\frac 1\sigma})}
$$
for some $h>0$ depdending on $r_2$, and thereby by
$r$ and $r_0$ only. This implies
\begin{multline}\label{Eq:PseudoDiffSemiNormEst}
|V_{\phi _2}g(x,\xi )| = 
|\scal {J(x,\xi ,\cdo )}{F}|
\\[1ex]
\lesssim
\nm f{S_{s;h}^\sigma}
\left (
\iint e^{r_2(|y|^{\frac 1s}+|\eta |^{\frac 1\sigma})}
e^{-(1+r_2)(|y|^{\frac 1s}+|\eta |^{\frac 1\sigma})}
\, dyd\eta
\right )
e^{- r(|x|^{\frac 1s}+|\xi |^{\frac 1\sigma})}
\\[1ex]
\asymp
\nm f{S_{s;h}^\sigma}
e^{- r(|x|^{\frac 1s}+|\xi |^{\frac 1\sigma})}
\end{multline}
which shows that $g\in \Sigma _s^\sigma (\rr d)$
in view of \cite[Proposition 2.1]{To15}.

\par

Since the topology of $\Sigma _s^\sigma (\rr d)$
is given by the semi-norms
$$
g\mapsto \sup _{x,\xi \in \rr d}|V_{\phi _2}g(x,\xi )
e^{r(|x|^{\frac 1s}+|\xi |^{\frac 1\sigma})}|
$$
it follows from \eqref{Eq:PseudoDiffSemiNormEst}
that $\op (a)$ is continuous on
$\Sigma _s^\sigma (\rr d)$.

\par

By duality it follows that $\op (a)$ is
uniquely extendable to a continuous map
on $(\Sigma _s^\sigma )'(\rr d)$.
\end{proof}

\par


\par

The following result follows by similar arguments as
in the previous proof. The
verifications are left for the reader.

\par

\begin{thm}\label{Thm:theorem3}
Let $A \in \MdR$, $s,\sigma >0$ be such that $s+\sigma \ge 1$ and
$(s,\sigma )\neq (\frac 12,\frac 12)$,
and let $a\in \Gamma _{s,\sigma}^{\sigma ,s} (\rr {2d})$.
Then $\op _A(a)$ is continuous from $\Sigma _s^\sigma (\rr d)$
to $\maclS _s^\sigma (\rr d)$, and from $(\maclS _s^\sigma )'(\rr d)$
to $(\Sigma _s^\sigma )'(\rr d)$.
\end{thm}

\par




\par




\par

\subsection{Compositions of pseudo-differential operators}

\par

Next we deduce algebraic properties of pseudo-differential
operators considered in Theorems \ref{Thm:theorem2}, \ref{Thm:theorem1}
and \ref{Thm:theorem3}. We recall that for pseudo-differential
operators with symbols in e.{\,}g. H{\"o}rmander classes, we have
$$
\op _0(a_1\wpr _0a_2) = \op _0(a_1)\circ \op _0(a_2),
$$
when
$$
a_1\wpr _0 a_2 (x,\xi ) \equiv \big ( e^{i\scal {D_\xi}{D_y}}
(a_1(x,\xi )a_2(y,\eta ) \big )
\Big \vert _{(y,\eta )=(x,\xi )}.
$$
More generally, if $A \in \MdR$ and $a_1\wpr _Aa_2$ is defined
by
\begin{equation}\label{SharpADef}
a_1\wpr _Aa_2  \equiv e^{i\scal {AD_\xi }{D_x}}\left ( \big ( e^{-i\scal {AD_\xi }{D_x}}a_1 \big )
\wpr _0 \big ( e^{-i\scal {AD_\xi }{D_x}}a_2 \big ) \right ) ,
\end{equation}
for $a_1$ and $a_2$ belonging to certain H{\"o}rmander symbol classes, then
it follows from the analysis in \cite{Ho} that
\begin{equation}\label{SharpAconjugation}
\op _A(a_1\wpr _Aa_2) = \op _A(a_1)\circ \op _A(a_2) 
\end{equation}
for suitable $a_1$ and $a_2$.

\par

We recall that the map $a\mapsto K_{a,A}$ is
a homeomorphism from
$\maclS _{s,\sigma}^{\sigma ,s}(\rr {2d})$
to $\maclS _{s}^{\sigma}(\rr {2d})$ and from
$\Sigma _{s,\sigma}^{\sigma ,s}(\rr {2d})$
to $\Sigma _{s}^{\sigma}(\rr {2d})$. It is also
obvious that the map
$$
(K_1,K_2)
\mapsto
\left (
(x,y) \mapsto (K_1\circ K_2)(x,y) =
\int _{\rr d}K_1(x,z)K_2(z,y)\, dz
\right )
$$
is sequentially continuous from
$\maclS _{s}^{\sigma}(\rr {2d})
\times
\maclS _{s}^{\sigma}(\rr {2d})$
to
$\maclS _{s}^{\sigma}(\rr {2d})$,
and from
$\Sigma _{s}^{\sigma}(\rr {2d})
\times
\Sigma _{s}^{\sigma}(\rr {2d})$
to
$\Sigma _{s}^{\sigma}(\rr {2d})$.
Here we have identified operators with their kernels.
Since
\begin{equation}\label{Eq:KernCompositins}
\begin{aligned}
(K_1\circ K_2\circ K_3)(x,y) &= \scal {K_2}{T_{K_1,K_3}(x,y,\cdo )}
\\[1ex]
\text{with}\quad
T_{K_1,K_3}(x,y,z_1,z_2) &= K_1(x,z_1)K_3(z_2,y)
\end{aligned}
\end{equation}
when $K_j\in L^2(\rr {2d})$, $j=1,2,3$, and that
$$
(K_1,K_2,K_3)
\mapsto 
\left (
(x,y)
\mapsto 
\scal {K_2}{T_{K_1,K_3}(x,y,\cdo )}
\right )
$$
is sequentially continuous from
$\maclS _s^\sigma (\rr {2d})\times
(\maclS _s^\sigma )'(\rr {2d})\times
\maclS _s^\sigma (\rr {2d})$ to $\maclS _s^\sigma (\rr {2d})$,
and from
$\Sigma _s^\sigma (\rr {2d})\times
(\Sigma _s^\sigma )'(\rr {2d})\times
\Sigma _s^\sigma (\rr {2d})$ to
$\Sigma _s^\sigma (\rr {2d})$,
the following result follows from these continuity results
and \eqref{SharpAconjugation}.

\par

\begin{prop}\label{Prop:GelShilAlgebras}
Let $A \in \MdR$, and let $s,\sigma >0$ be such that
$s+\sigma \ge 1$. Then the following is true:
\begin{enumerate}
\item the map $(a_1,a_2)\mapsto a_1\wpr _Aa_2$ is continuous
from $\maclS _{s,\sigma}^{\sigma ,s}(\rr {2d})
\times \maclS _{s,\sigma}^{\sigma ,s}(\rr {2d})$ to
$\maclS _{s,\sigma}^{\sigma ,s}(\rr {2d})$;

\vrum

\item the map $(a_1,a_2)\mapsto a_1\wpr _Aa_2$ is continuous
from $\Sigma _{s,\sigma}^{\sigma ,s}(\rr {2d})
\times
\Sigma _{s,\sigma}^{\sigma ,s}(\rr {2d})$ to
$\Sigma _{s,\sigma}^{\sigma ,s}(\rr {2d})$;

\vrum

\item the map $(a_1,a_2,a_3)\mapsto a_1\wpr _Aa_2\wpr _Aa_3$
from $\maclS _{s,\sigma}^{\sigma ,s}(\rr {2d})
\times \maclS _{s,\sigma}^{\sigma ,s}(\rr {2d})
\times \maclS _{s,\sigma}^{\sigma ,s}(\rr {2d})$ to
$\maclS _{s,\sigma}^{\sigma ,s}(\rr {2d})$ extends
uniquely to a sequentially continuous and associative
map from $\maclS _{s,\sigma}^{\sigma ,s}(\rr {2d})
\times (\maclS _{s,\sigma}^{\sigma ,s})'(\rr {2d})
\times \maclS _{s,\sigma}^{\sigma ,s}(\rr {2d})$ to
$\maclS _{s,\sigma}^{\sigma ,s}(\rr {2d})$;

\vrum

\item the map $(a_1,a_2,a_3)\mapsto a_1\wpr _Aa_2\wpr _Aa_3$
from $\Sigma _{s,\sigma}^{\sigma ,s}(\rr {2d})
\times \Sigma _{s,\sigma}^{\sigma ,s}(\rr {2d})
\times \Sigma _{s,\sigma}^{\sigma ,s}(\rr {2d})$ to
$\Sigma _{s,\sigma}^{\sigma ,s}(\rr {2d})$ extends
uniquely to a sequentially continuous and associative
map from $\Sigma _{s,\sigma}^{\sigma ,s}(\rr {2d})
\times (\Sigma _{s,\sigma}^{\sigma ,s})'(\rr {2d})
\times \Sigma _{s,\sigma}^{\sigma ,s}(\rr {2d})$ to
$\Sigma _{s,\sigma}^{\sigma ,s}(\rr {2d})$;
\end{enumerate}
\end{prop}

\par

We have the following corresponding algebra result for
$\Gamma _{s,\sigma}^{\sigma ,s;0}$ and related symbol classes.

\par

\begin{thm}\label{Thm:GammaAlgebras}
Let $A \in \MdR$, and let $s,\sigma >0$ be such that
$s+\sigma \ge 1$. Then the following is true:
\begin{enumerate}
\item the map {\rm{(1)}} in Proposition \ref{Prop:GelShilAlgebras}
extends uniquely to a continuous map from
$\Gamma _{s,\sigma ;0}^{\sigma ,s}(\rr {2d})\times
\Gamma _{s,\sigma ;0}^{\sigma ,s}(\rr {2d})$ to
$\Gamma _{s,\sigma ;0}^{\sigma ,s}(\rr {2d})$, and
from
$\Gamma _{s,\sigma ;0}^{\sigma ,s;0}(\rr {2d})\times
\Gamma _{s,\sigma ;0}^{\sigma ,s;0}(\rr {2d})$ to
$\Gamma _{s,\sigma ;0}^{\sigma ,s;0}(\rr {2d})$;

\vrum

\item if in addition $(s,\sigma )\neq (\frac 12,\frac 12)$,
the map {\rm{(2)}} in Proposition \ref{Prop:GelShilAlgebras}
extends uniquely to a continuous map from
$\Gamma _{s,\sigma }^{\sigma ,s;0}(\rr {2d})\times
\Gamma _{s,\sigma }^{\sigma ,s;0}(\rr {2d})$ to
$\Gamma _{s,\sigma}^{\sigma ,s;0}(\rr {2d})$, and
from
$\Gamma _{s,\sigma }^{\sigma ,s}(\rr {2d})\times
\Gamma _{s,\sigma }^{\sigma ,s;0}(\rr {2d})$
or from
$\Gamma _{s,\sigma }^{\sigma ,s;0}(\rr {2d})\times
\Gamma _{s,\sigma }^{\sigma ,s}(\rr {2d})$ 
to
$\Gamma _{s,\sigma}^{\sigma ,s}(\rr {2d})$.
\end{enumerate}
\end{thm}

\par

\begin{proof}
We prove only the first assertion in (2). The other statements
follow by similar arguments and are left for the reader.

\par

By Theorem \ref{ThmCalculiTransf2} it suffices to consider the
case when $A=0$. Let $\phi _1,\phi _2,\phi _3\in \Sigma _s^\sigma (\rr d)
\setminus 0$,
$a_j\in \Gamma _{s,\sigma}^{\sigma ,s;0}(\rr {2d})$,
$j=1,2$, and let $K$ be the kernel of $\op _0(a_1)\circ \op _0(a_2)$.
By Proposition \ref{Prop:prop2B} we need to prove that for some
$r>0$,
\begin{equation}\label{Eq:STFTProdKernel}
|V_{\phi _1\otimes \phi _3}K(x,y,\xi ,\eta )|
\lesssim
e^{r(|x|^{\frac 1s}+|\eta |^{\frac 1\sigma})
-h(|\xi +\eta |^{\frac 1\sigma}+|x-y|^{\frac 1s})}
\end{equation}
for every $h>0$.

\par

Therefore, let $h>0$ be arbitrarily chosen but fixed, and let
$K_j$ be the kernel of $\op _0(a_j)$, $j=1,2$,
\begin{align*}
F_1(x,y,\xi ,\eta ) &= V_{\phi _1\otimes \phi _2}K_1(x,y,\xi ,\eta ),
\\[1ex]
F_2(x,y,\xi ,\eta ) &= V_{\phi _2\otimes \phi _3}K_2(x,y, -\xi ,\eta )
\intertext{and}
G(x,y,\xi ,\eta ) &= V_{\phi _1\otimes \phi _3}K(x,y,\xi ,\eta ).
\end{align*}
Then
\begin{equation}\label{Eq:GIntegralFormula}
G(x,y,\xi ,\eta ) = \iint _{\rr {2d}}F_1(x,z,\xi ,\zeta )
F_2(z,y,\zeta ,\eta )\, dzd\zeta
\end{equation}
by Moyal's identity (cf. proof of Theorem \ref{Thm:theorem1}).
Since $a_j\in \Gamma _{s,\sigma}^{\sigma ,s;0}(\rr {2d})$ we have for
some $r>0$ that
\begin{align*}
|F_1(x,y,\xi ,\eta )|
&\lesssim
e^{r(|x|^{\frac 1s}+|\eta |^{\frac 1\sigma})
-h_1(|\xi +\eta |^{\frac 1\sigma}+|x-y|^{\frac 1s})}
\intertext{and}
|F_2(x,y,\xi ,\eta )|
&\lesssim
e^{r(|x|^{\frac 1s}+|\eta |^{\frac 1\sigma})
-h_1(|\xi -\eta |^{\frac 1\sigma}+|x-y|^{\frac 1s})}
\end{align*}
for every $h_1>0$.
By combining this with \eqref{Eq:GIntegralFormula} we get for some
$r>0$,
\begin{equation}\label{Eq:GEst}
|G(x,y,\xi ,\eta )|
\lesssim
\iint _{\rr {2d}}
e^{\fy _{r,h_2}(x,y,z,\xi ,\eta ,\zeta ) + \psi _{r,h_2}(x,y,z,\xi ,\eta ,\zeta )}
\, dzd\zeta ,
\end{equation}
where $h_2\ge 2ch+cr$,
\begin{align*}
\fy _{r,h}
(x,y,z,\xi ,\eta ,\zeta ) &= r(|x|^{\frac 1s} +|\zeta |^{\frac 1\sigma})
-h(|\zeta -\eta |^{\frac 1\sigma}+|y-z|^{\frac 1s})
\intertext{and}
\psi _{r,h}
(x,y,z,\xi ,\eta ,\zeta ) &= r(|z|^{\frac 1s} +|\eta |^{\frac 1\sigma})
-h(|\xi +\zeta |^{\frac 1\sigma}+|x-z|^{\frac 1s}).
\end{align*}
and $c\ge 1$ is chosen such that
$$
|x+y|^{\frac 1s}\le c(|x|^{\frac 1s}+|y|^{\frac 1s})
\quad \text{and}\quad
|\xi +\eta |^{\frac 1\sigma}\le c(|\xi |^{\frac 1\sigma}+|\eta|
^{\frac 1\sigma}),
\quad x,y,\xi ,\eta \in \rr d.
$$

\par

Then
\begin{align*}
\fy _{r,h_2}
(x,y,z,\xi ,\eta ,\zeta ) &\le cr(|x|^{\frac 1s} +|\eta |^{\frac 1\sigma})
-(h_2-cr)(|\zeta -\eta |^{\frac 1\sigma}+|y-z|^{\frac 1s})
\\[1ex]
&\le cr(|x|^{\frac 1s} +|\eta |^{\frac 1\sigma})
-2ch(|\zeta -\eta |^{\frac 1\sigma}+|y-z|^{\frac 1s})
\intertext{and}
\psi _{r,h_2}
(x,y,z,\xi ,\eta ,\zeta ) &\le cr(|x|^{\frac 1s} +|\eta |^{\frac 1\sigma})
-2ch(|\xi +\zeta |^{\frac 1\sigma}+|x-z|^{\frac 1s}).
\end{align*}
This gives
\begin{multline*}
\fy _{r,h_2}(x,y,z,\xi ,\eta ,\zeta )
+ \psi _{r,h_2}(x,y,z,\xi ,\eta ,\zeta )
\\[1ex]
\le
2cr(|x|^{\frac 1s} +|\eta |^{\frac 1\sigma})
-2ch(|\xi +\zeta |^{\frac 1\sigma}
+ |\zeta -\eta |^{\frac 1\sigma}
+|x-z|^{\frac 1s}
+|y-z|^{\frac 1s}).
\end{multline*}
Since
\begin{multline*}
-2ch(|\xi +\zeta |^{\frac 1\sigma}
+ |\zeta -\eta |^{\frac 1\sigma}
+|x-z|^{\frac 1s}
+|y-z|^{\frac 1s})
\\[1ex]
\le
-h(|\xi +\eta |^{\frac 1\sigma}
+|x-y|^{\frac 1s})
- ch(|\xi +\zeta |^{\frac 1\sigma}
+ |\zeta -\eta |^{\frac 1\sigma}
+|x-z|^{\frac 1s}
+|y-z|^{\frac 1s})
\\[1ex]
\le
-h(|\xi +\eta |^{\frac 1\sigma}
+|x-y|^{\frac 1s})
- ch(|\xi +\zeta |^{\frac 1\sigma}
+|x-z|^{\frac 1s})
\end{multline*}
we get by combining these estimates with \eqref{Eq:GEst}
that
\begin{multline*}
|G(x,y,\xi ,\eta )|
\lesssim
\iint _{\rr {2d}}
e^{2cr(|x|^{\frac 1s} +|\eta |^{\frac 1\sigma})
-h(|\xi +\eta |^{\frac 1\sigma}
+|x-y|^{\frac 1s})
- ch(|\xi +\zeta |^{\frac 1\sigma}
+|x-z|^{\frac 1s})}
\, dzd\zeta ,
\\[1ex]
\asymp
e^{2cr(|x|^{\frac 1s} +|\eta |^{\frac 1\sigma})
-h(|\xi +\eta |^{\frac 1\sigma}
+|x-y|^{\frac 1s})}.
\end{multline*}
Since $r>0$ is fixed and $h>0$ can be chosen arbitrarily,
the result follows.
\end{proof}

\par

\begin{thm}\label{Composition}
Let $A \in \MdR$, $s,\sigma >0$ be such that $s+\sigma \ge 1$,
and let $\omega_j \in \mathscr P_{s,\sigma}(\rr {2d})$, $j=1,2$.
Then the following is true:
\begin{enumerate}
\item the map $(a_1,a_2)\mapsto a_1\wpr _Aa_2$ from
$\Sigma _{s,\sigma}^{\sigma ,s}(\rr {2d})
\times
\Sigma _{s,\sigma}^{\sigma ,s}(\rr {2d})$
to $\Sigma _{s,\sigma}^{\sigma ,s}(\rr {2d})$ is uniquely
extendable to a continuous map from
$\Gamma _{(\omega _1)}^{\sigma ,s;0}(\rr {2d})
\times
\Gamma _{(\omega _2)}^{\sigma ,s;0}(\rr {2d})$ to
$\Gamma _{(\omega _1\omega _2)}^{\sigma ,s;0})(\rr {2d})$;

\vrum

\item if in addition $\omega_j \in
\mathscr P_{s,\sigma}^0(\rr {2d})$, $j=1,2$,
then the map $(a_1,a_2)\mapsto a_1\wpr _Aa_2$ from
$\maclS _{s,\sigma}^{\sigma ,s}(\rr {2d})
\times
\maclS _{s,\sigma}^{\sigma ,s}(\rr {2d})$
to $\maclS _{s,\sigma}^{\sigma ,s}(\rr {2d})$ is uniquely
extendable to a continuous map from
$\Gamma _{(\omega _1)}^{\sigma ,s}(\rr {2d})
\times
\Gamma _{(\omega _2)}^{\sigma ,s}(\rr {2d})$ to
$\Gamma _{(\omega _1\omega _2)}^{\sigma ,s}(\rr {2d})$.
\end{enumerate}
\end{thm}

\par

\begin{proof}
We may assume that $A=0$ by Theorem \ref{Thm:CalculiTransf}. We
only prove (2). The assertion (1) follows by similar arguments and
is left for the reader.

\par

Let
\begin{align*}
F_{a_1,a_2}(x_1,x_2,\xi _1,\xi _2) &= a_1(x_1,\xi _1)a_2(x_2,\xi _2)
\intertext{and}
\omega (x_1,x_2,\xi _1,\xi _2) &= \omega _1(x_1,\xi _1)\omega _2(x_2,\xi _2).
\end{align*}
By the definitions it follows that the map $T_1$ which takes
$(a_1,a_2)$ into $F_{a_1,a_2}$ is continuous from
$\Gamma _{(\omega _1)}^{\sigma ,s}(\rr {2d})\times
\Gamma _{(\omega _2)}^{\sigma ,s}(\rr {2d})$ to
$\Gamma _{(\omega )}^{\sigma ,s}(\rr {4d})$.

\par

Theorem \ref{Thm:CalculiTransfbis} declare that the map $T_2$
which takes $F(x_1,x_2,\xi _1,\xi _2)$ to 
$e^{i\scal {D_{\xi _1}}{D_{x_2}}}F(x_1,x_2,\xi _1,\xi _2)$ is continuous
on $\Gamma _{(\omega )}^{\sigma ,s}(\rr {4d})$. Hence, if $T_3$ is
the trace operator which takes $F(x_1,x_2,\xi _1,\xi _2)$ into
$F_0(x,\xi ) \equiv F(x,x,\xi,\xi )$, Proposition A.2 shows that
$T\equiv T_3\circ T_2\circ T_1$ is continuous from
$\Gamma _{(\omega _1)}^{\sigma ,s}(\rr {2d})\times
\Gamma _{(\omega _2)}^{\sigma ,s}(\rr {2d})$ to
$\Gamma _{(\omega _1\omega _2)}^{\sigma ,s}(\rr {2d})$.

\par

By \cite[Theorem 18.1.8]{Ho} we have $T(a_1,a_2)= a_1\wpr _0a_2$ when
$a_1,a_2\in \Sigma _{s,\sigma}^{\sigma ,s}(\rr {2d})$. If instead
$a_k\in \Gamma
_{(\omega _k)}^{\sigma ,s}(\rr {2d})$, $k=1,2$, then we take
$T(a_1,a_2)$ as the definition of $a_1\wpr _0a_2$. By the continuity of $T$ it follows that
$(a_1,a_2)\mapsto a_1\wpr _0a_2$ is continuous from
$\Gamma _{(\omega _1)}^{\sigma ,s}(\rr {2d})\times
\Gamma _{(\omega _2)}^{\sigma ,s}(\rr {2d})$ to
$\Gamma _{(\omega _1\omega _2)}^{\sigma ,s}(\rr {2d})$.

\par

Since $\Gamma _{(\omega _k)}^{\sigma ,s}(\rr {2d})\subseteq
\Gamma _{s,\sigma ;0}^{\sigma ,s}(\rr {2d})$, we get
$\op _0(a_1\wpr _0 a_2) = \op _0(a_1)\circ \op _0(a_2)$ and
that $a_1\wpr _0a_2$ is uniquely defined as an element in
$\Gamma _{s,\sigma ;0}^{\sigma ,s}(\rr {2d})$, in view of
Theorem \ref{Thm:GammaAlgebras}. Hence $a_1\wpr _0a_2$
is uniquely defined in $\Gamma _{(\omega _1\omega _2)}^{\sigma ,s}
(\rr {2d})$, since all these symbol classes are subspaces of
$C^\infty (\rr {2d})$. This gives the result.
\end{proof}

\par

\par

\section*{Appendix A}

\par

In what follows we prove some auxiliary results on continuity of Gevrey symbol
classes.

\par

\renewcommand{\rubrik}{Proposition A.1}

\par

\begin{tom}
Let $\sigma ,s >0$,
$\omega \in \mascP _{s ,\sigma}^0(\rr {2d})$,
$a \in \Gamma _{s,\sigma}^{\sigma ,s}(\rr {2d})$ and let $a_{\ep}=\phi (\ep \cdo )a$,
$\ep >0$, where $\phi \in \maclS _{s,\sigma}^{\sigma ,s}(\rr {2d})$
satisfies $\phi (0)=1$.
Then the following is true:
\begin{enumerate}
\item $a_\ep \to a$ in $\Gamma _{s,\sigma}^{\sigma,s}(\rr {2d})$ as $\ep \to 0^+$;

\vrum

\item If in addition $a \in \Gamma _{s,\sigma ;0}^{\sigma ,s}(\rr {2d})$, then
$a_\ep \to a$ in $\Gamma _{s,\sigma ;0}^{\sigma,s}(\rr {2d})$ as $\ep \to 0^+$;

\vrum

\item If in addition $a \in \Gamma _{s,\sigma}^{\sigma ,s ;0}(\rr {2d})$
($a \in \Gamma _{s,\sigma ;0}^{\sigma ,s;0}(\rr {2d})$) and $\phi \in
\Sigma _{s,\sigma}^{\sigma ,s}(\rr {2d})$, then
$a_\ep \to a$ in $\Gamma _{s,\sigma}^{\sigma,s ;0}(\rr {2d})$ 
(in $\Gamma _{s,\sigma ;0}^{\sigma,s;0}(\rr {2d})$) as $\ep \to 0^+$.
\end{enumerate}
\end{tom}

\par

\begin{proof}
We only prove (2). The other assertions follow by similar arguments and is
left for the reader.
Since
$\phi \in \maclS _{s,\sigma}^{\sigma ,s}(\rr {2d})$, there are constants
$C,r_0,h_0>0$ which are independent of $\delta >0$ such that
$$
|\partial _x^\alpha \partial _\xi ^\beta \phi _\ep (x,\xi )|
\le Ch_0^{|\alpha +\beta |}\ep ^{|\alpha +\beta |}\alpha !^{\sigma}\beta !^s
e^{-r_0\ep (|x|^{\frac 1s}+|\xi |^{\frac 1\sigma})}.
$$
For conveniency we also let
\begin{align*}
\nm a{s,\sigma , h,r,\alpha ,\beta } &\equiv
\sup _{x,\xi \in \rr d}\left (
\frac {e^{-r (|x| ^{\frac 1s}+|\xi |^{\frac 1\sigma})} |\partial _x^\alpha
\partial _\xi ^\beta a(x,\xi )|}
{\alpha !^\sigma \beta !^s h^{|\alpha |}}
\right )
\intertext{and}
\nm a{s, \sigma , h,r} &\equiv \sup _{\alpha ,\beta \in \nn d} \nm a{s,\sigma , h ,r,\alpha ,\beta }.
\end{align*}

\par

By Leibniz rule we get
$$
\partial _x^\alpha \partial _\xi ^\beta a_\ep
=
\sum _{\gamma \le \alpha} \sum _{\delta \le \beta} {\alpha \choose \gamma}
{\beta \choose \delta}
\partial _x^\gamma \partial _\xi ^\delta \phi _\ep \cdot 
\partial _x^{\alpha -\gamma} \partial _\xi ^{\beta -\delta} a.
$$
Hence, if $r>0$ is arbitrary and $h \ge h_0$ is chosen such that $\nm a{s, \sigma , h,r}<\infty$,
we get
\begin{equation*}
\nm {a_\ep -a}{s,\sigma ,h ,r,\alpha ,\beta}
\le
J_{1}(s,\sigma ,h ,r,\ep ) + J_{2}(s,\sigma ,h ,r,\ep ,\alpha ,\beta ),
\end{equation*}
where
\begin{alignat*}{1}
J_{1}(s,\sigma ,&h ,r,\ep ) = \sup _{\alpha ,\beta \in \nn d}
\frac 
{\nm {e^{-r_0|\cdo |^2}(1-\phi (\ep \cdo ))\partial _x^\alpha \partial _\xi ^\beta a}{L^\infty}}
{\alpha !^{\sigma}\beta !^sh^{|\alpha +\beta |}}
\intertext{and}
J_{2}(s,\sigma ,&h ,r,\ep ,\alpha ,\beta )
\\
&=
(\alpha !^{\sigma}\beta !^sh^{|\alpha +\beta |})^{-1}
\sum  {\alpha \choose \gamma}{\beta \choose \delta}
\nm {\partial _x^\gamma \partial _\xi ^\delta \phi _\delta}{L^\infty}
\nm {\partial _x^{\alpha -\gamma}\partial _\xi ^{\beta -\delta}a}{L^\infty} ,
\end{alignat*}
where the last sum is taken over all $\gamma \in \nn d$ and $\delta \in \nn d$ such that
$\gamma \le \alpha$, $\delta \le \beta$ and $(\gamma ,\delta )\neq (0,0)$.

\par

Evidently, since $a\in  \Gamma _{s,\sigma}^{\sigma ,s ;\ep}(\rr {2d})$, it
follows by
straight-forward
estimates that $J_{1}(s,\sigma ,\ep _1,h, \delta )\to 0$ as
$\delta \to 0^+$. For $J_{2}(s,\sigma ,\ep _1,\delta ,\alpha )$ we have
\begin{multline*}
J_{2}(s,\sigma ,\ep_1,\delta ,\alpha) \lesssim (\alpha !^{s}\ep_1^{|\alpha |})^{-1}
\sum _{0\neq  \gamma \le \alpha} {\alpha \choose \gamma}
\ep_0^{|\gamma |}\delta ^{\frac {|\gamma |}2}\gamma !^{s}
\ep^{|\alpha -\gamma |}(\alpha -\gamma )!^{s}
\\[1ex]
\le
\frac {\delta ^{\frac 12}}{\ep_1^{|\alpha |}} \sum _{\gamma \le \alpha}
{\alpha \choose \gamma} \ep_0^{|\gamma |}\ep^{|\alpha -\gamma |}
= \delta ^{\frac 12}\left ( \frac {\ep_0+\ep}{\ep_1}\right )^{|\alpha |} \le \delta ^{\frac 12},
\end{multline*}
provided $\ep_1$ is chosen larger than $\ep_0+\ep$. This gives the result.
\end{proof}

\par

The next result concerns mapping properties of $\Gamma _{s,\sigma}^{\sigma ,s}$
spaces under trace operators.

\par

\par

\renewcommand{\rubrik}{Proposition A.2}

\begin{tom}
Let $\omega$ be a weight on $\rr {4d}$, $\omega _0(x,\xi ) =
\omega (x,x,\xi ,\xi )$ when $x,\xi \in \rr d$, $s,\sigma >0$
be such that $s+\sigma \ge 1$. Then the
trace map which takes
\begin{align*}
\rr {4d}\ni (x,y,\xi ,\eta) &\mapsto F(x,y,\xi ,\eta )
\intertext{to}
\rr {2d}\ni (x,\xi ) &\mapsto F(x,x,\xi ,\xi )
\end{align*}
is linear and continuous from $\Gamma _{(\omega )}^{\sigma ,s}(\rr {4d})$
into $\Gamma _{(\omega _0)}^{\sigma ,s}(\rr {2d})$. The same holds true
with $\Gamma _{(\omega )}^{\sigma ,s;0}$ and $\Gamma
_{(\omega _0)}^{\sigma ,s;0}$ in place of
$\Gamma _{(\omega )}^{\sigma ,s}$ and $\Gamma
_{(\omega _0)}^{\sigma ,s}$, respectively, at each occurrence.
\end{tom}

\par

Proposition A.2 follows by similar arguments as in the proof
of Lemma \ref{Lemma:theorem2}, using the Leibnitz type rule
$$
\partial _x^\alpha \partial _\xi ^\beta (F(x,x,\xi,\xi ))
=
\sum _{\gamma \le \alpha }\sum _{\delta \le \beta}
{\alpha \choose \gamma}{\beta \choose \delta}
(\partial _1^{\alpha -\gamma}\partial _2^{\beta -\delta}\partial _3^\gamma
\partial _4^\delta F)(x,x,\xi ,\xi).
$$
The details are left for the reader.

\par

\end{document}